\documentclass[10pt,letterpaper]{amsart}%
\usepackage{amsfonts}
\usepackage{amsmath}
\usepackage{amssymb}
\usepackage{color}
\usepackage{graphicx}%
\newtheorem{theorem}{Theorem}[section]
\theoremstyle{plain}

\newtheorem{lemma}[theorem]{Lemma}

\newtheorem{proposition}[theorem]{Proposition}
\newtheorem{remark}{Remark}

\numberwithin{equation}{section}

\begin{document}

\title[Sharp mixed norm Fourier extension]{Sharp radial--angular mixed norm Fourier extension and global stability on spheres}

\author{Zewei Xu}
\address[Zewei Xu]{Laboratory of Mathematics and Complex Systems (Ministry of Education), School of Mathematical Sciences, Beijing Normal University, Beijing, 100875, China}
\email{zwxu@mail.bnu.edu.cn}


\begin{abstract}

We study the sharp Fourier extension inequality from $L^2(\mathbb S^{d-1})$ to the radial--angular mixed space $L^q_{\mathrm{rad}}L^2_{\mathrm{ang}}$.  Our work is motivated by and builds upon the foundational analysis of Carneiro, Oliveira e Silva and Sousa \cite{COS}, who initiated the sharp mixed-norm spherical restriction problem from the viewpoint of spherical-harmonic decomposition.  Their reduction revealed that the extremal problem is governed by a competition among the individual harmonic modes, and their results established degree-zero extremizers in several important ranges of exponents, including even integer exponents and the endpoint case.  Subsequent developments, including the work of Ciccone and Sousa \cite{CS}, extended the known regions of modal dominance and further clarified the behavior of low harmonic degrees.  These results naturally suggest a more complete question: whether the zeroth mode dominates all higher modes throughout the whole admissible range and whether one can quantify the spectral gap between them.  The present paper addresses this question by establishing a complete modal ordering result.  Building on their spherical-harmonic reduction, the sharp-constant problem is reduced to the weighted Bessel moments
\[
I_{\nu,k}(q)
:=
\int_0^\infty
|J_{\nu+k}(r)|^q r^{2\nu+1-\nu q}\,dr,
\qquad
\nu:=\frac d2-1.
\]
For every integer $d\ge2$ and every finite admissible exponent $q>2d/(d-1)$, we prove
\[
I_{\nu,k}(q)<I_{\nu,0}(q),
\qquad k\ge1,
\]
and hence constants are the unique nonzero extremizers of the sharp mixed-norm extension inequality.  In dimensions $d\ge3$ we prove the stronger strict ordering
\[
I_{\nu,0}(q)>I_{\nu,1}(q)>I_{\nu,2}(q)>\cdots.
\]
The proof combines two different mechanisms.  In dimensions $d\ge3$, the comparison between the first two modes is obtained from a direct Bessel identity, while the remaining infinitely many comparisons are reduced to a continuous variation of the Bessel order and a quantitative positive-kernel argument.  The two-dimensional circle is genuinely different: the endpoint of the critical range produces a common divergent tail, and therefore the cancellation of this tail has to be preserved throughout the argument.  We resolve this difficulty by combining exact Hankel asymptotics with a finite remainder comparison.

The modal ordering also yields the optimal quadratic global stability inequality.  For every finite admissible $q$,
\[
C_{d,q}^2\|f\|_2^2
-\|\widehat{f\sigma}\|_{L^q_{\mathrm{rad}}L^2_{\mathrm{ang}}}^2
\ge
\Gamma_{d,q}\|f-P_0f\|_2^2,
\]
with the best possible coefficient.  When $d\ge3$, this coefficient is explicitly
\[
\Gamma_{d,q}
=(2\pi)^d
\left[
I_{\nu,0}(q)^{2/q}-I_{\nu,1}(q)^{2/q}
\right].
\]
Finally, the endpoint $q=\infty$ is handled separately because the integral
representation of the finite moments is no longer available.  Writing
\[
\mu_k=\sup_{r\ge0}|r^{-\nu}J_{\nu+k}(r)|^2,
\]
we prove the strict hierarchy $\mu_0>\mu_1>\mu_2>\cdots$ and derive the
corresponding sharp global stability inequality for the radial essential supremum norm.
\end{abstract}

\maketitle

\section{Introduction}

Let $\sigma$ denote the unnormalized surface measure on $\mathbb S^{d-1}$ and write
\[
\omega_{d-1}:=\sigma(\mathbb S^{d-1})
=\frac{2\pi^{d/2}}{\Gamma(d/2)}.
\]
For $f\in L^2(\mathbb S^{d-1},\sigma)$, we write the Fourier transform of the measure $f\sigma$ as
\begin{equation}\label{eq:intro-extension}
\widehat{f\sigma}(x)
:=
\int_{\mathbb S^{d-1}}e^{-ix\cdot\omega}f(\omega)\,d\sigma(\omega),
\qquad x\in\mathbb R^d.
\end{equation}
We consider the radial--angular mixed norm
\begin{equation}\label{eq:intro-mixed-norm}
\|F\|_{L^q_{\mathrm{rad}}L^2_{\mathrm{ang}}}
:=
\left[
\int_0^\infty
\left(
\int_{\mathbb S^{d-1}}|F(r\theta)|^2\,d\sigma(\theta)
\right)^{q/2}
r^{d-1}\,dr
\right]^{1/q},
\end{equation}
with the usual essential-supremum interpretation when $q=\infty$.

Mixed radial--angular restriction estimates go back to Vega \cite{VegaThesis,Vega}; see also C\'ordoba \cite{Cordoba} and C\'ordoba--Latorre \cite{CordobaLatorre}.  For the sphere, the mixed-norm extension estimate
\[
\|\widehat{f\sigma}\|_{L^q_{\mathrm{rad}}L^2_{\mathrm{ang}}}
\le C_{d,q}\|f\|_{L^2(\mathbb S^{d-1})}
\]
holds precisely in the strong-type range
\[
q>\frac{2d}{d-1}.
\]
The sharp constant and the extremizers are governed by a simple but nontrivial spectral question.  Indeed, if $P_k$ denotes the orthogonal projection onto the spherical harmonics of degree $k$, then the classical Fourier--Bessel formula \cite[Chapter IV, Eq.~(2.19) and Theorem~3.10]{SteinWeiss} gives
\begin{equation}\label{eq:intro-FB}
\widehat{Y_k\sigma}(r\theta)
=(2\pi)^{d/2}(-i)^k r^{-\nu}J_{\nu+k}(r)Y_k(\theta),
\qquad
\nu:=\frac d2-1.
\end{equation}
As observed in \cite{COS}, this identity and angular orthogonality separate the harmonic degrees exactly; consequently the sharp-constant problem reduces to comparing
\begin{equation}\label{eq:intro-I}
I_{\nu,k}(q)
:=
\int_0^\infty
|J_{\nu+k}(r)|^q r^{2\nu+1-\nu q}\,dr,
\qquad k=0,1,2,\ldots.
\end{equation}

Carneiro, Oliveira e Silva and Sousa \cite{COS} initiated the sharp mixed-norm spherical restriction analysis by reducing the problem to the comparison of the harmonic modes and establishing the basic modal framework that underlies the present work.  A central insight of their work is that, after the spherical-harmonic decomposition, the sharp constant is determined by the largest among a family of weighted Bessel moments associated with different harmonic degrees.  They proved that the constant mode is extremal in several important regimes, including admissible even integer exponents and the endpoint $q=\infty$, and their analysis raised the natural problem of understanding the full range of exponents for which the zeroth mode remains dominant.  In particular, the modal comparison problem left open the possibility of a more detailed ordering theory describing not only the largest mode but also the relative size of the higher modes.  Subsequent work of Ciccone and Sousa \cite{CS} enlarged some known ranges of global maximality and provided further information on the role of low-degree spherical harmonics.  The results of these works form the starting point of the present paper: rather than replacing the previous reduction, we use this spectral framework and pursue the finer question of complete modal ordering and the resulting quantitative stability estimates.  This mixed-norm problem belongs to the broader theory of sharp Fourier restriction and extension inequalities developed, among others, in \cite{Tomas,Foschi2007,Foschi2015,CarneiroOliveira,ChristShao,ChristShaoExt,Shao2016,CarneiroFoschiOST,FoschiOliveira,CNS}; see also \cite{Stein1993} for general restriction background.  Thus the central question is whether the zero mode dominates all higher spherical harmonics throughout the full admissible range.

\begin{theorem}[Carneiro--Oliveira e Silva--Sousa]\label{thm:COS-background}
Let $d\ge2$, put
\[
\nu:=\frac d2-1,
\]
and let
\[
q>\frac{2d}{d-1}
\]
be finite.  Define the modal sequence
\[
I_{\nu,k}(q)
:=
\int_0^\infty
|J_{\nu+k}(r)|^q r^{2\nu+1-\nu q}\,dr,
\qquad k\ge0 .
\]
The following statements hold.

\begin{enumerate}
\item The sequence $\{I_{\nu,k}(q)\}_{k\ge0}$ satisfies
\[
\lim_{k\to\infty}I_{\nu,k}(q)=0,
\]
and the sharp constant is determined by
\[
C_{d,q}
=(2\pi)^{d/2}
\max_{k\ge0}I_{\nu,k}(q)^{1/q}.
\]
Moreover, if $\ell$ satisfies
\[
I_{\nu,\ell}(q)=\max_{k\ge0}I_{\nu,k}(q),
\]
then the extremizers are contained in the spherical harmonic space
$\mathcal H_\ell^d$.

\item If $q$ is an admissible even integer, then
\[
I_{\nu,0}(q)=\max_{k\ge0}I_{\nu,k}(q),
\]
and the constant functions are the unique extremizers.

\item For a fixed dimension $d$, define
\[
\mathcal A_d=
\{q\in(2d/(d-1),\infty):
I_{\nu,0}(q)>I_{\nu,k}(q)\ \text{for all }k\ge1\}.
\]
Then $\mathcal A_d$ is open. Moreover, it contains all admissible even
integers, and one may define
\[
q_0(d):=
\inf\{r:\frac{2d}{d-1}<r<\infty,\ (r,\infty)\subset\mathcal A_d\}.
\]
\end{enumerate}
\end{theorem}

We see from this theorem that the sharp mixed-norm problem is reduced to
determining the maximal element of the modal sequence
$\{\Lambda_{d,q}(k)\}_{k\ge0}$.  Following the work of
Carneiro, Oliveira e Silva and Sousa \cite{COS}, subsequent studies have
further improved the known range of exponents associated with the threshold
$q_0(d)$.  These developments motivate the present work, where we investigate
the finer ordering structure among all spherical harmonic modes and derive the
corresponding quantitative stability estimates.

Our first theorem gives a complete answer.

\begin{theorem}\label{thm:main-extension}
Let $d\ge2$ be an integer and put $\nu:=d/2-1$.  For every finite
\[
q>\frac{2d}{d-1},
\]
one has
\begin{equation}\label{eq:main-moment-comparison}
I_{\nu,k}(q)<I_{\nu,0}(q),
\qquad k\ge1.
\end{equation}
Hence the sharp constant in
\[
\|\widehat{f\sigma}\|_{L^q_{\mathrm{rad}}L^2_{\mathrm{ang}}}
\le C_{d,q}\|f\|_2
\]
is
\begin{equation}\label{eq:main-Cdq}
C_{d,q}
=(2\pi)^{d/2}
\left(
\int_0^\infty
|J_{d/2-1}(r)|^q
r^{d-1-q(d/2-1)}\,dr
\right)^{1/q},
\end{equation}
and the nonzero extremizers are exactly the almost-everywhere constant functions.

If $d\ge3$, then the stronger ordering
\begin{equation}\label{eq:full-ordering-intro}
I_{\nu,0}(q)>I_{\nu,1}(q)>I_{\nu,2}(q)>\cdots
\end{equation}
holds for every finite admissible $q$.
\end{theorem}

The proof has two parts.  In dimensions $d\ge3$, we first separate the base pair.  A scalar flux identity and a reflected Wronskian comparison around consecutive zeros prove $I_{\nu,0}(q)>I_{\nu,1}(q)$ directly for every admissible exponent.  Once this bridge is available, every remaining order satisfies $\alpha\ge\nu+1\ge3/2$.  Landau's order-variation identity then reduces the remaining comparisons to a quantitative sign analysis of the order derivative.  The associated positive kernel is not introduced merely as a formal representation: its monotonicity provides the negative contribution needed to dominate the possible positive part.  More precisely, the small-radius contribution is controlled through a logarithmic compensation estimate together with a Gaussian improvement of the Bessel bound, while on a carefully chosen oscillatory interval the Hankel asymptotic phase yields a uniform negative contribution.  The parameter estimates show that this negative part dominates the entire error, which proves all later comparisons and hence the full ordering \eqref{eq:full-ordering-intro}.

The circle requires a different argument.  Here $\nu=0$ and the critical exponent is $q=4$, where the individual fourth moments are divergent.  Thus one cannot start from the endpoint and interpolate.  We instead compare the common tails before estimating them: the first few modes are treated by an explicit scalar remainder, while the higher modes are controlled through the exact Hankel phase.  This gives $I_k(q)<I_0(q)$ for every $k\ge1$ and every $q>4$.

Once the strict modal comparison is known, the global stability problem becomes transparent.  Put
\[
u_k(r):=r^{-\nu}J_{\nu+k}(r),
\qquad
\phi_k(r):=u_k(r)^2,
\qquad
\lambda_k:=I_{\nu,k}(q)^{2/q},
\qquad
K_d:=(2\pi)^d,
\]
and let $P_0$ denote the orthogonal projection onto the constant functions.  For finite admissible $q$, define
\[
D_{d,q}(f)
:=C_{d,q}^2\|f\|_2^2
-\|\widehat{f\sigma}\|_{L^q_{\mathrm{rad}}L^2_{\mathrm{ang}}}^2,
\qquad
\lambda_*:=\max_{k\ge1}\lambda_k.
\]
The maximum exists because $\lambda_k\to0$. This decay of the higher harmonic modes follows from the modal analysis developed in \cite{COS}; we use it here as a consequence of the spherical-harmonic reduction rather than as a new ingredient.

\begin{theorem}\label{thm:intro-global-stability}
For every $d\ge2$, every finite admissible $q$, and every
$f\in L^2(\mathbb S^{d-1})$,
\begin{equation}\label{eq:intro-global-stability}
D_{d,q}(f)
\ge
\Gamma_{d,q}\|f-P_0f\|_2^2,
\qquad
\Gamma_{d,q}=K_d(\lambda_0-\lambda_*)>0.
\end{equation}
The coefficient is optimal.  Equality holds exactly for constants and for functions contained in a single maximizing nonzero harmonic space.  Here ``maximizing'' means that the degree $j\ge1$ of that space satisfies
\[
\lambda_j=\max_{k\ge1}\lambda_k.
\]

If $d\ge3$, then $\lambda_*=\lambda_1$ and
\begin{equation}\label{eq:intro-global-explicit}
\Gamma_{d,q}
=(2\pi)^d
\left[
I_{\nu,0}(q)^{2/q}-I_{\nu,1}(q)^{2/q}
\right].
\end{equation}
In this case the nonconstant equality cases are precisely the first spherical harmonics.
\end{theorem}

The endpoint $q=\infty$ retains the same spectral structure, although the
radial $L^q$ norm is replaced by a supremum and the moment comparison must be
replaced by a comparison of the maximal Bessel amplitudes.  Set
\[
\mu_k
:=
\sup_{r\ge0}|r^{-\nu}J_{\nu+k}(r)|^2.
\]
We prove
\[
\mu_0>\mu_1>\mu_2>\cdots,
\qquad
\mu_0=\frac{2^{-2\nu}}{\Gamma(\nu+1)^2},
\]
and obtain the following sharp endpoint stability estimate.

\begin{theorem}\label{thm:intro-endpoint-stability}
For every $d\ge2$ and every $f\in L^2(\mathbb S^{d-1})$,
\[
C_{d,\infty}=\omega_{d-1},
\]
and
\begin{equation}\label{eq:intro-endpoint-stability}
C_{d,\infty}^2\|f\|_2^2
-
\|\widehat{f\sigma}\|_{L^\infty_{\mathrm{rad}}L^2_{\mathrm{ang}}}^2
\ge
(2\pi)^d(\mu_0-\mu_1)\|f-P_0f\|_2^2.
\end{equation}
The coefficient is optimal, and equality holds exactly for constants and pure first spherical harmonics.
\end{theorem}

The paper is organized as follows.  Section 2 collects the preliminary spherical-harmonic reduction and the basic Bessel identities used throughout.  Section 3 proves Theorem~\ref{thm:main-extension} in a proof-driven order: in dimensions $d\ge3$ the discrete modal problem is first reduced to continuous variation in the Bessel order, the required estimates are introduced as the positive and negative parts of the order derivative are controlled, and the single zero-to-first-mode bridge is handled at the end; the circle is then treated by a low-mode argument and an exact-phase critical-tail comparison.  Section 4 proves the sharp finite-exponent stability theorem from an exact deficit decomposition, Section 5 treats the endpoint $q=\infty$ through the ordering of the modal suprema, and Section 6 records the routine analytic justifications and normalization remarks used in the preceding arguments.

\section{Preliminaries}\label{sec:preliminaries}

The purpose of this section is only to fix the spectral reduction and the basic
Bessel notation.  All more specialized auxiliary estimates will be introduced
later, precisely at the point where they enter the proofs.

Let $\mathcal H_k^d$ be the space of spherical harmonics of degree $k$ on $\mathbb S^{d-1}$, and let $P_k$ be the corresponding orthogonal projection.  Put
\[
a_k:=\|P_kf\|_2^2,
\qquad
\sum_{k=0}^\infty a_k=\|f\|_2^2.
\]
For $Y_k\in\mathcal H_k^d$, the Fourier--Bessel identity \cite[Chapter IV, Eq.~(2.19) and Theorem~3.10]{SteinWeiss} (see also the reduction in \cite[Section~2]{COS}) is
\begin{equation}\label{eq:FB}
\widehat{Y_k\sigma}(r\theta)
=(2\pi)^{d/2}(-i)^k r^{-\nu}J_{\nu+k}(r)Y_k(\theta).
\end{equation}
Thus orthogonality on the sphere gives
\begin{equation}\label{eq:spectral-identity}
\|\widehat{f\sigma}\|_{L^q_{\mathrm{rad}}L^2_{\mathrm{ang}}}^2
=(2\pi)^d
\left\|
\sum_{k=0}^\infty a_k r^{-2\nu}J_{\nu+k}(r)^2
\right\|_{L^{q/2}(r^{d-1}dr)}.
\end{equation}
For $q>2$, Minkowski's inequality yields
\begin{align}
\|\widehat{f\sigma}\|_{L^q_{\mathrm{rad}}L^2_{\mathrm{ang}}}^2
&\le
(2\pi)^d\sum_{k=0}^\infty a_k I_{\nu,k}(q)^{2/q}.
\label{eq:Minkowski-modal}
\end{align}
Once \eqref{eq:main-moment-comparison} is known, \eqref{eq:Minkowski-modal} immediately gives the sharp constant \eqref{eq:main-Cdq}.  Moreover, if $a_k>0$ for some $k\ge1$, the comparison with the degree-zero moment is strict for that nonzero mode.

For completeness, one can justify \eqref{eq:spectral-identity} directly for arbitrary $f\in L^2$.  Let $f_N=\sum_{k=0}^N P_kf$.  Since
\[
\|f_N-f\|_1\le \omega_{d-1}^{1/2}\|f_N-f\|_2\to0,
\]
we have $\widehat{f_N\sigma}\to\widehat{f\sigma}$ uniformly in $x$.  Angular Parseval gives the finite modal sum at each radius, and monotone convergence then gives \eqref{eq:spectral-identity}.

At the endpoint $q=\infty$, Cauchy--Schwarz gives
\begin{equation}\label{eq:endpoint-CS}
\|\widehat{f\sigma}(r\,\cdot)\|_{L^2(\mathbb S^{d-1})}
\le
\omega_{d-1}\|f\|_2,
\qquad r\ge0.
\end{equation}
Constants attain equality as $r\to0^+$.  The equality case is rigid.  Indeed, $r\mapsto \|\widehat{f\sigma}(r\,\cdot)\|_2$ is continuous, tends to zero at infinity, and therefore attains its supremum.  At $r=0$, equality in \eqref{eq:endpoint-CS} forces $f$ to be constant.  If equality were attained at some $r>0$, then for almost every $\theta$ one would have
\[
f(\omega)=c_\theta e^{ir\theta\cdot\omega}
\]
almost everywhere in $\omega$.  Taking two distinct directions would force $e^{ir(\theta_1-\theta_2)\cdot\omega}$ to be constant on the sphere, which is impossible.

Finally, the lower endpoint of the admissible range is sharp.  For the constant input, the classical fixed-order expansion \cite[Chapter~VII]{Watson} (see also \cite[Chapter~7]{Lebedev})
\[
J_\nu(r)
=
\sqrt{\frac{2}{\pi r}}
\cos\left(r-\frac{\pi\nu}{2}-\frac\pi4\right)
+O(r^{-3/2})
\]
shows that the radial integral diverges when $q\le2d/(d-1)$.

We collect the estimates needed in both the higher-dimensional and the circle arguments.  All powers on the positive real axis are taken with the real branch.

For $\alpha\ge0$, we use the classical ascending series, differential equation, recurrence relations and integral representation recorded, for example, in \cite[Chapters~II--III]{Watson} and \cite[Chapter~5]{Lebedev}:
\begin{equation}\label{eq:Bessel-series}
J_\alpha(r)
=
\frac{(r/2)^\alpha}{\Gamma(\alpha+1)}
\sum_{m=0}^\infty
\frac{(-r^2/4)^m}{m!(\alpha+1)_m},
\qquad
(\alpha+1)_m:=\frac{\Gamma(\alpha+m+1)}{\Gamma(\alpha+1)}.
\end{equation}
The Bessel equation and the derivative recurrences are
\begin{equation}\label{eq:Bessel-ODE}
J_\alpha''+\frac1rJ_\alpha'
+\left(1-\frac{\alpha^2}{r^2}\right)J_\alpha=0,
\end{equation}
\begin{equation}\label{eq:Bessel-recurrences}
J'_{\alpha-1}
=\frac{\alpha-1}{r}J_{\alpha-1}-J_\alpha,
\qquad
J_\alpha'
=J_{\alpha-1}-\frac\alpha rJ_\alpha.
\end{equation}
For integer $k$ we also use
\begin{equation}\label{eq:Bessel-integral}
J_k(r)
=\frac1{2\pi}\int_{-\pi}^{\pi}e^{i(k\theta-r\sin\theta)}\,d\theta,
\end{equation}
which in particular gives $|J_k(r)|\le1$.

The circle argument uses Landau's uniform order bound \cite{LandauJLMS,LandauEJDE} in the relaxed form
\begin{equation}\label{eq:Landau-amplitude}
|J_\alpha(r)|<\frac7{10}\alpha^{-1/3},
\qquad \alpha>0,\ r>0.
\end{equation}
We write
\begin{equation}\label{eq:m-A}
m(p):=\frac1\pi\int_0^\pi|\cos u|^p\,du,
\qquad
A(p):=2\left(\frac{2}{\pi}\right)^{p/2}m(p).
\end{equation}
In particular,
\begin{equation}\label{eq:cos-moments}
m(3)=\frac4{3\pi},
\qquad
m(4)=\frac38,
\qquad
m(5)=\frac{16}{15\pi}>\frac13.
\end{equation}

These facts form the common background for all three main theorems.  We now
turn directly to the proof of Theorem~\ref{thm:main-extension}; from this point
on, each auxiliary lemma is stated only when it is needed.

\section{Proof of Theorem~\ref{thm:main-extension}}\label{sec:proof-main}

We divide the proof according to the dimension.  In dimensions $d\ge3$ we
first embed the discrete modal sequence into a continuous family in the Bessel
order and reduce the desired ordering to the sign of an order derivative.  The
auxiliary estimates are then introduced only when the positive and negative
parts of that derivative have to be controlled; the single zero-to-first-mode
comparison is handled separately at the end.  The circle has a different
critical tail and is treated in the second subsection.

\subsection{The case $d\ge3$}\label{sec:higher-dimensional}

Fix $d\ge3$, put $\nu:=d/2-1\ge1/2$, and let
\[
q>\frac{4(\nu+1)}{2\nu+1}.
\]
For the rest of this subsection we write
\begin{equation}\label{eq:beta-higher-dimensional}
\beta:=2\nu+1-\nu q,
\qquad
M_\alpha(q,\beta):=\int_0^\infty r^\beta |J_\alpha(r)|^q\,dr.
\end{equation}
Then
\[
I_{\nu,k}(q)=M_{\nu+k}(q,\beta).
\]
The point of introducing the continuous order $\alpha$ is that the discrete
ordering in Theorem~\ref{thm:main-extension} can be obtained from a sign
condition for a single derivative.  We make this reduction explicit before
introducing any of the estimates used to prove that sign.

\begin{proposition}[Continuous-order reduction]\label{prop:continuous-order-reduction}
Suppose that
\[
\partial_\alpha M_\alpha(q,\beta)<0,
\qquad \alpha\ge\nu+1.
\]
Then
\[
I_{\nu,1}(q)>I_{\nu,2}(q)>I_{\nu,3}(q)>\cdots.
\]
If, in addition, $I_{\nu,0}(q)>I_{\nu,1}(q)$, then the full strict ordering
\eqref{eq:full-ordering-intro} follows.  If the derivative is already negative
for every $\alpha\ge\nu$, then the full ordering follows without a separate
first-mode argument.
\end{proposition}

\begin{proof}
For every integer $k\ge1$,
\[
I_{\nu,k+1}(q)-I_{\nu,k}(q)
=
\int_{\nu+k}^{\nu+k+1}
\partial_\alpha M_\alpha(q,\beta)\,d\alpha.
\]
The assertions are therefore immediate.
\end{proof}

Thus the higher-dimensional proof has a clear division of labor.  The main
part is to prove continuous-order decay after the first mode; only at the end
will we return to the single missing link between orders $\nu$ and $\nu+1$.
To differentiate in the order, we use the positive kernel
\[
K_0(z)=\int_0^\infty e^{-z\cosh t}\,dt,\qquad z>0,
\]
so that $K_0$ is positive and decreasing.  For $\alpha,r>0$ define

\begin{equation}\label{eq:Malpha-Aalpha}
\mathcal A_\alpha(r)
:=
\int_0^\infty K_0(2r\sinh t)e^{-2\alpha t}\,dt.
\end{equation}
The positivity and radial monotonicity of $\mathcal A_\alpha$ follow directly from this representation.  Indeed, the integrand in \eqref{eq:Malpha-Aalpha} is positive, so $\mathcal A_\alpha(r)>0$.  Moreover, using $K_0'(x)=-K_1(x)$ and
\[
K_1(x)=\int_0^\infty e^{-x\cosh s}\cosh s\,ds>0,
\qquad x>0,
\]
we may differentiate under the integral on compact positive-radius intervals to obtain
\[
\mathcal A_\alpha'(r)
=-2\int_0^\infty
\sinh t\,K_1(2r\sinh t)e^{-2\alpha t}\,dt<0.
\]
Thus $r\mapsto\mathcal A_\alpha(r)$ is strictly decreasing on $(0,\infty)$.

\begin{lemma}\label{lem:order-derivative}
For $\alpha>0$ and $r>0$,
\begin{equation}\label{eq:order-variation}
\partial_\alpha J_\alpha(r)
=
r\mathcal A_\alpha'(r)J_\alpha(r)
-2r\mathcal A_\alpha(r)J_\alpha'(r).
\end{equation}
If $q>1$, $\beta+q\alpha>-1$ and $\beta<q/2-1$, then
\begin{equation}\label{eq:order-moment-derivative}
\partial_\alpha M_\alpha(q,\beta)
=
\int_0^\infty
r^\beta |J_\alpha(r)|^qD_\alpha(r)\,dr,
\end{equation}
where
\begin{equation}\label{eq:Dalpha}
D_\alpha(r)
=(q+2)r\mathcal A_\alpha'(r)
+2(\beta+1)\mathcal A_\alpha(r).
\end{equation}
In particular, $\partial_\alpha M_\alpha(q,\beta)<0$ whenever $\beta\le-1$.
\end{lemma}

\begin{proof}
Identity \eqref{eq:order-variation} is Landau's order-variation formula,
specialized to $F\equiv1$ and $C_\alpha=J_\alpha$; see
\cite[Eq.~(7), pp.~150--151]{LandauEJDE} and \cite{LandauJLMS}.
Thus, with $g(r):=|J_\alpha(r)|^q$,
\[
\partial_\alpha g
=qr\mathcal A_\alpha'(r)g-2r\mathcal A_\alpha(r)g'(r).
\]
After multiplication by $r^\beta$ and integration, one integration by parts
gives
\[
\begin{aligned}
\partial_\alpha M_\alpha(q,\beta)
&=q\int_0^\infty r^{\beta+1}\mathcal A_\alpha'g\,dr
   -2\int_0^\infty r^{\beta+1}\mathcal A_\alpha g'\,dr\\
&=\int_0^\infty r^\beta g(r)
\left[(q+2)r\mathcal A_\alpha'(r)
      +2(\beta+1)\mathcal A_\alpha(r)\right]\,dr,
\end{aligned}
\]
which is \eqref{eq:order-moment-derivative}--\eqref{eq:Dalpha}.  The
differentiation under the integral sign and the vanishing of the boundary
terms are justified in Section~\ref{app:limits}.

Finally, $\mathcal A_\alpha>0$ and $\mathcal A_\alpha'<0$.  Hence
$D_\alpha(r)<0$ for every $r>0$ when $\beta\le-1$, and the last assertion
follows.
\end{proof}

\begin{remark}
Lemma~\ref{lem:order-derivative} is the mechanism that turns the discrete comparison of adjacent spherical-harmonic modes into a continuous comparison in the Bessel order.  The hypotheses on $\beta$ are exactly those needed to justify differentiation under the integral sign and the integration by parts, while the zeros of $J_\alpha$ create no additional terms because $|x|^q$ is $C^1$ for $q>1$.
\end{remark}

For the spherical parameters in \eqref{eq:beta-higher-dimensional}, the
integrability hypotheses in Lemma~\ref{lem:order-derivative} are precisely the
usual convergence conditions for the Bessel moments; the routine
justifications are recorded in Section~\ref{app:limits}.  The range
\[
\beta\le-1
\]
is therefore settled immediately.  In this range, Lemma~\ref{lem:order-derivative} gives
\[
\partial_\alpha M_\alpha(q,\beta)<0,
\qquad \alpha\ge\nu,
\]
and Proposition~\ref{prop:continuous-order-reduction} proves the whole chain.
No further kernel estimate is needed in this range.

We may henceforth assume
\[
\beta>-1,
\qquad b:=\beta+1>0.
\]
Now the second term in $D_\alpha$ is positive, so the integrand in
\eqref{eq:order-moment-derivative} need not have a fixed sign.  Rather than
trying to locate its exact sign changes, we use a positive-versus-negative
balance.  Our first task is to find a cutoff $R_0$ such that
\[
D_\alpha(r)\le0,
\qquad r\ge R_0.
\]
Once such a cutoff is available, define
\[
P_\alpha
:=
\int_0^{R_0}r^\beta|J_\alpha(r)|^q[D_\alpha(r)]_+\,dr.
\]
Then $P_\alpha$ contains all possible positive contribution.  On the
negative tail we are free to retain only a convenient interval
$I_\alpha\subset[R_0,\infty)$ and define
\[
N_\alpha
:=
\int_{I_\alpha}r^\beta|J_\alpha(r)|^q[-D_\alpha(r)]\,dr.
\]
Then
\begin{equation}\label{eq:PN-reduction}
\partial_\alpha M_\alpha(q,\beta)
\le P_\alpha-N_\alpha.
\end{equation}
Consequently, the continuous-order problem has been reduced to the concrete
inequality $P_\alpha<N_\alpha$.  The rest of the argument is organized around
exactly these two quantities: first we force a negative tail, then estimate
$P_\alpha$ from above and $N_\alpha$ from below.

To force the tail sign, write
\[
\tau:=\frac{2b}{q+2},
\qquad
L_\alpha(r):=-\frac{r\mathcal A_\alpha'(r)}{\mathcal A_\alpha(r)}.
\]
Then
\begin{equation}\label{eq:Dalpha-Lalpha}
D_\alpha(r)
=(q+2)\mathcal A_\alpha(r)(\tau-L_\alpha(r)).
\end{equation}
Thus it remains to obtain a lower bound for $L_\alpha$.  The comparison is
based on the classical Chebyshev integral inequality, which we record now in
exactly the form needed for this purpose.

\begin{lemma}[Chebyshev integral inequality]\label{lem:positive-covariance}
Let $\mu$ be a probability measure on a real interval.  If $f$ and $g$ are nondecreasing and the relevant products are integrable, then
\begin{equation}\label{eq:covariance}
\int fg\,d\mu-
\int f\,d\mu\int g\,d\mu
=
\frac12\iint
(f(x)-f(y))(g(x)-g(y))\,d\mu(x)d\mu(y)
\ge0.
\end{equation}
Consequently, if $R\ge0$ is nondecreasing and $0<\int R\,d\mu<\infty$, then
\[
\frac{\int Rg\,d\mu}{\int R\,d\mu}
\ge
\int g\,d\mu.
\]
\end{lemma}

This is the classical Chebyshev integral inequality for similarly ordered functions; see, for example, \cite[Chapter~10]{MitrinovicPecaricFink}.  The double-integral identity in \eqref{eq:covariance} is the standard covariance representation, and the last assertion follows by applying it with $f=R$ and dividing by $\int R\,d\mu$.  These observations prove Lemma~\ref{lem:positive-covariance}.

We now develop the lower bound needed for the tail sign in
\eqref{eq:Dalpha-Lalpha}.  The quantity $L_\alpha$ is the negative
logarithmic derivative of $\mathcal A_\alpha$ with respect to the logarithmic
radial variable:
\[
L_\alpha(r)
=
-\frac{d}{d\log r}\log \mathcal A_\alpha(r).
\]
Since $\mathcal A_\alpha$ is represented by a positive kernel, the
Chebyshev inequality allows us to compare this logarithmic derivative with
that of a one-variable model.  To formulate the comparison, put
\begin{equation}\label{eq:F-ell-chi}
F(z):=\int_0^\infty K_0(2zu)e^{-2u}\,du,
\qquad
\ell(z):=-\frac{zF'(z)}{F(z)},
\qquad
\chi_\alpha:=\sqrt{\alpha^2+\frac14}.
\end{equation}
The positive integral representation of $K_0$ first reduces $F$ to an elementary integral.  Indeed, by Tonelli's theorem,
\[
\begin{aligned}
F(z)
&=\int_0^\infty\int_0^\infty
 e^{-2zu\cosh t}e^{-2u}\,dt\,du\\
&=\frac12\int_0^\infty\frac{dt}{1+z\cosh t}.
\end{aligned}
\]
To evaluate the last integral, set
\[
v:=\tanh(t/2).
\]
Then
\[
\cosh t=\frac{1+v^2}{1-v^2},
\qquad
dt=\frac{2\,dv}{1-v^2},
\]
and $t\in[0,\infty)$ corresponds to $v\in[0,1)$.  Thus
\[
F(z)
=\int_0^1\frac{dv}{(1+z)+(z-1)v^2}.
\]
If $0<z<1$, this is
\[
\frac1{\sqrt{1-z^2}}
\operatorname{artanh}\sqrt{\frac{1-z}{1+z}}.
\]
The identity
\[
2\operatorname{artanh}\sqrt{\frac{1-z}{1+z}}
=\operatorname{arcosh}\frac1z
\]
follows because both sides are positive and have hyperbolic cosine $1/z$.
If $z>1$, the same rational integral gives
\[
\frac1{\sqrt{z^2-1}}
\arctan\sqrt{\frac{z-1}{z+1}},
\]
Since both sides lie in $(0,\pi/2)$, the double-angle formula for cosine gives
\[
2\arctan\sqrt{\frac{z-1}{z+1}}
=\arccos\frac1z.
\]
At $z=1$ the integral is simply $\int_0^1(2)^{-1}dv=1/2$.  Hence
\begin{equation}\label{eq:F-explicit}
F(z)
=\frac12\int_0^\infty\frac{dt}{1+z\cosh t}
=
\begin{cases}
\displaystyle
\frac{\operatorname{arcosh}(1/z)}{2\sqrt{1-z^2}},&0<z<1,\\[7pt]
\displaystyle \frac12,&z=1,\\[7pt]
\displaystyle
\frac{\arccos(1/z)}{2\sqrt{z^2-1}},&z>1.
\end{cases}
\end{equation}

The kernel $\mathcal A_\alpha$ itself is classical in Landau's order-variation formula \cite{LandauEJDE,LandauJLMS}, but the quantitative comparison with the model $F$ and $\ell$ below is the form required for the present argument.

\begin{lemma}\label{lem:kernel-comparison}
For all $\alpha,r>0$,
\begin{equation}\label{eq:kernel-comparison}
L_\alpha(r):=-\frac{r\mathcal A_\alpha'(r)}{\mathcal A_\alpha(r)}
\ge
\ell(r/\chi_\alpha),
\end{equation}
and
\begin{equation}\label{eq:A-lower-model}
\mathcal A_\alpha(r)
\ge
\frac1{\chi_\alpha}F(r/\chi_\alpha)
\ge
\frac{\pi}{4(r+\chi_\alpha)}.
\end{equation}
The function $\ell$ is increasing on $(0,\infty)$ and satisfies
\begin{equation}\label{eq:ell-bounds}
\ell(1/\sqrt2)>\frac35,
\qquad
\ell(z)\ge1-\frac{2}{3(z+1)},
\qquad z\ge1.
\end{equation}
\end{lemma}

\begin{proof}
We divide the proof into three steps.

\emph{Step 1: comparison of the kernel with the model weight.}
Put $y:=\sinh t$ in \eqref{eq:Malpha-Aalpha}.  Since $dt=dy/\sqrt{1+y^2}$,
\[
\mathcal A_\alpha(r)
=
\int_0^\infty K_0(2ry)w_\alpha(y)\,dy,
\qquad
w_\alpha(y)
:=
\frac{e^{-2\alpha\operatorname{arsinh}y}}{\sqrt{1+y^2}}.
\]
A logarithmic derivative gives
\[
-\frac{w_\alpha'(y)}{w_\alpha(y)}
=
\frac{2\alpha}{\sqrt{1+y^2}}+\frac{y}{1+y^2}.
\]
Write $y=\tan u$, $0\le u<\pi/2$.  Then
\[
-\frac{w_\alpha'}{w_\alpha}
=2\alpha\cos u+\sin u\cos u
\le2\alpha\cos u+\sin u.
\]
By Cauchy--Schwarz,
\[
2\alpha\cos u+\sin u
\le\sqrt{4\alpha^2+1}
=2\chi_\alpha.
\]
Therefore
\[
\frac d{dy}\log\bigl(w_\alpha(y)e^{2\chi_\alpha y}\bigr)
=\frac{w_\alpha'}{w_\alpha}+2\chi_\alpha\ge0.
\]
Thus
\[
R_\alpha(y):=w_\alpha(y)e^{2\chi_\alpha y}
\]
is nondecreasing.  Since $R_\alpha(0)=1$, we have
\[
w_\alpha(y)\ge e^{-2\chi_\alpha y}.
\]
Consequently
\[
\begin{aligned}
\mathcal A_\alpha(r)
&\ge\int_0^\infty K_0(2ry)e^{-2\chi_\alpha y}\,dy\\
&=\frac1{\chi_\alpha}
\int_0^\infty K_0\!\left(2\frac r{\chi_\alpha}u\right)e^{-2u}\,du\\
&=\frac1{\chi_\alpha}F(r/\chi_\alpha).
\end{aligned}
\]
For the second lower bound in \eqref{eq:A-lower-model}, use the representation
\[
F(z)=\frac12\int_0^\infty\frac{dt}{1+z\cosh t}.
\]
Since $\cosh t\ge1$,
\[
1+z\cosh t\le(1+z)\cosh t,
\]
so
\[
F(z)
\ge\frac1{2(1+z)}\int_0^\infty\operatorname{sech}t\,dt
=\frac\pi{4(1+z)}.
\]
Substituting $z=r/\chi_\alpha$ proves
\[
\mathcal A_\alpha(r)\ge\frac\pi{4(r+\chi_\alpha)}.
\]

\emph{Step 2: comparison of logarithmic derivatives.}
Define
\[
\psi(x):=\frac{xK_1(x)}{K_0(x)}.
\]
The identities $K_0'=-K_1$ and $(xK_1)'=-xK_0$ imply
\[
\psi'(x)
=x\left[\left(\frac{K_1(x)}{K_0(x)}\right)^2-1\right].
\]
To see that this is positive, use the positive integral representations
\[
K_0(x)=\int_0^\infty e^{-x\cosh t}\,dt,
\qquad
K_1(x)=\int_0^\infty e^{-x\cosh t}\cosh t\,dt.
\]
Hence $K_1(x)>K_0(x)>0$ for $x>0$, and so $\psi'(x)>0$.

Now write
\[
w_\alpha(y)=R_\alpha(y)e^{-2\chi_\alpha y}
\]
and normalize
\[
K_0(2ry)e^{-2\chi_\alpha y}\,dy
\]
to a probability measure $\mu$.  Differentiating $\mathcal A_\alpha$ in $r$ and using $K_0'=-K_1$ gives
\[
-r\mathcal A_\alpha'(r)
=
\int_0^\infty K_0(2ry)\psi(2ry)w_\alpha(y)\,dy.
\]
Therefore
\[
L_\alpha(r)
=
\frac{\int R_\alpha(y)\psi(2ry)\,d\mu(y)}{\int R_\alpha\,d\mu}.
\]
Both $R_\alpha(y)$ and $\psi(2ry)$ are nondecreasing in $y$.  Lemma \ref{lem:positive-covariance} therefore gives
\[
L_\alpha(r)
\ge
\int\psi(2ry)\,d\mu(y).
\]
It remains to identify the last average.  Put $z:=r/\chi_\alpha$ and change variables $u:=\chi_\alpha y$.  Since
\[
F'(z)=-2\int_0^\infty uK_1(2zu)e^{-2u}\,du,
\]
we find
\[
\int\psi(2ry)\,d\mu(y)
=
\frac{-zF'(z)}{F(z)}
=\ell(z)
=\ell(r/\chi_\alpha).
\]
This proves \eqref{eq:kernel-comparison}.

\emph{Step 3: monotonicity and explicit lower bounds for $\ell$.}
For $0<z<1$ write $z:=\operatorname{sech}t$.  From \eqref{eq:F-explicit},
\[
F(\operatorname{sech}t)=\frac t2\coth t.
\]
Using $\ell(z)=-d\log F/d\log z$ gives
\[
\ell(\operatorname{sech}t)
=\frac{\coth t}{t}-\operatorname{csch}^2t.
\]
Similarly, for $z>1$ write $z:=\sec t$, $0<t<\pi/2$.  Then
\[
F(\sec t)=\frac t2\cot t,
\qquad
\ell(\sec t)=\csc^2t-\frac{\cot t}{t}.
\]
Differentiation yields
\begin{align}
\frac d{dt}\ell(\operatorname{sech}t)
&=
-\frac{\cosh t}{t^2\sinh^3t}
\bigl(\sinh^2t+t\tanh t-2t^2\bigr),\label{eq:ell-hyp-derivative}\\
\frac d{dt}\ell(\sec t)
&=
\frac{\cos t}{t^2\sin^3t}
\bigl(\sin^2t+t\tan t-2t^2\bigr).\label{eq:ell-trig-derivative}
\end{align}
The first bracket is nonnegative because
\[
\sinh^2t\ge t^2+\frac{t^4}{3},
\qquad
\tanh t\ge t-\frac{t^3}{3}.
\]
For the second bracket we use
\[
\sin^2t\ge t^2-\frac{t^4}{3},
\qquad
\tan t\ge t+\frac{t^3}{3}.
\]
For instance,
\[
\frac d{dt}\left(\tan t-t-\frac{t^3}{3}\right)
=\tan^2t-t^2\ge0,
\]
while
\[
\frac d{dt}\left(\tanh t-t+\frac{t^3}{3}\right)
=t^2-\tanh^2t\ge0.
\]
Hence $t\mapsto\ell(\operatorname{sech}t)$ is nonincreasing, while $t\mapsto\ell(\sec t)$ is nondecreasing.  Since $\operatorname{sech}t$ decreases and $\sec t$ increases, $\ell$ is increasing on both $(0,1)$ and $(1,\infty)$.  From the rational representation
\[
F(z)=\int_0^1\frac{dv}{1+z+(z-1)v^2}
\]
one sees directly that $F$ is $C^1$ at $z=1$; indeed $F(1)=1/2$ and $F'(1)=-1/3$, so $\ell(1)=2/3$.  Thus $\ell$ is increasing on all of $(0,\infty)$.

At $z=1/\sqrt2$, the hyperbolic parameter satisfies $\sinh t=1$, hence
\[
\ell(1/\sqrt2)=\frac{\sqrt2}{\operatorname{arsinh}1}-1>\frac35;
\]
the last numerical inequality is elementary.  Finally, for $z=\sec t\ge1$, the bound
\[
\ell(z)\ge1-\frac{2}{3(z+1)}
\]
is equivalent, after elementary simplification, to
\[
3\sin t\le t(2+\cos t).
\]
If $H(t):=t(2+\cos t)-3\sin t$, then
\[
H(0)=H'(0)=H''(0)=0,
\qquad
H'''(t)=t\sin t\ge0.
\]
Therefore $H''\ge0$, then $H'\ge0$, and finally $H\ge0$ on $[0,\pi/2)$.  This completes the proof of \eqref{eq:ell-bounds}.
\end{proof}

We can now fix the negative tail for the actual spherical parameters.  Since
$b>0$ and
\[
b=2-\nu(q-2),
\]
we have $2<q<6$.  The convergence condition at infinity gives $b<q/2$, while
$\nu\ge1/2$ gives
\[
b\le3-\frac q2.
\]
Moreover, for every order relevant after the first mode,
\begin{equation}\label{eq:first-post-bridge-order}
\alpha\ge\nu+1=\frac{q-b}{q-2}\ge\frac32.
\end{equation}
Consequently
\[
\tau=\frac{2b}{q+2}
\le
\min\left\{\frac{q}{q+2},\frac{6-q}{q+2}\right\}
\le\frac35.
\]
Set $\sigma:=3/4$ and recall $\chi_\alpha=(\alpha^2+1/4)^{1/2}$.  Since
$\alpha\ge3/2$,
\[
\frac{\sigma\alpha}{\chi_\alpha}
\ge\frac{9}{4\sqrt{10}}
>\frac1{\sqrt2}.
\]
Lemma~\ref{lem:kernel-comparison}, the monotonicity of $\ell$, and
\eqref{eq:ell-bounds} therefore give
\[
L_\alpha(r)
\ge\ell(r/\chi_\alpha)
\ge\ell(\sigma\alpha/\chi_\alpha)
>\ell(1/\sqrt2)
>\frac35
\ge\tau,
\qquad r\ge\sigma\alpha.
\]
Hence
\begin{equation}\label{eq:post-bridge-tail-negative}
D_\alpha(r)<0,
\qquad r\ge\sigma\alpha.
\end{equation}
This settles the tail-sign step in the positive--negative balance and fixes
the cutoff as
\[
R_0:=\sigma\alpha=\frac{3\alpha}{4}.
\]

With the cutoff fixed, we next control the possible positive contribution
$P_\alpha$, beginning with $[D_\alpha]_+$ on $(0,R_0)$.  The following
logarithmic compensation estimate is tailored to exactly this step.

\begin{lemma}\label{lem:compensation}
Let $\alpha>0$, $q>1$ and $\beta>-1$, put
\[
b=\beta+1>0,
\qquad
\tau:=\frac{2b}{q+2},
\]
and let $D_\alpha$ be defined by \eqref{eq:Dalpha}.  If $R_0>0$ and
$D_\alpha(R_0)\le0$, then
\begin{equation}\label{eq:compensation}
[D_\alpha(r)]_+
\le
\frac b\alpha\log\frac{R_0}{r},
\qquad 0<r<R_0.
\end{equation}
\end{lemma}

\begin{proof}
Let
\[
F_\alpha(z):=\alpha\mathcal A_\alpha(\alpha z),
\qquad
H_\alpha(z):=-zF_\alpha'(z).
\]
We first rewrite these functions in a scale adapted to the order.  Starting from the definition of $\mathcal A_\alpha$ and putting $u=\alpha t$ gives
\[
F_\alpha(z)
=\int_0^\infty
K_0\!\left(2\alpha z\sinh(u/\alpha)\right)e^{-2u}\,du.
\]
Set
\[
x:=2\alpha z\sinh(u/\alpha).
\]
Since $z\partial_zx=x$ and $K_0'=-K_1$, differentiation under the integral gives
\[
H_\alpha(z)
=
\int_0^\infty xK_1(x)e^{-2u}\,du.
\]
The differentiation under the integral sign is justified by the Dominated Convergence Theorem, since
$xK_1(x)$ is bounded on $(0,\infty)$, as will be seen from the
positive representation below.
The elementary positive representation
\[
xK_1(x)=\int_0^\infty e^{-\sqrt{x^2+y^2}}\,dy
\]
follows from $K_1(x)=\int_0^\infty e^{-x\cosh t}\cosh t\,dt$ after the change of variables $y=x\sinh t$.  Since $\sqrt{x^2+y^2}\ge y$,
\[
0<xK_1(x)\le\int_0^\infty e^{-y}\,dy=1.
\]
Therefore
\[
0<H_\alpha(z)\le\int_0^\infty e^{-2u}\,du=\frac12.
\]

We now differentiate with respect to the logarithmic variable $\log z$.  Recall that $d/d\log z=z\,d/dz$.  By definition,
\[
\frac d{d\log z}\bigl(\tau F_\alpha(z)\bigr)
=-\tau H_\alpha(z).
\]
Moreover, using $(xK_1(x))'=-xK_0(x)$ and again $z\partial_zx=x$,
\[
\frac d{d\log z}H_\alpha(z)
=
-\int_0^\infty x^2K_0(x)e^{-2u}\,du.
\]
This differentiation is also justified by the Dominated Convergence Theorem, since
$x^2K_0(x)$ is bounded on $(0,\infty)$.
Hence
\[
\frac d{d\log z}(\tau F_\alpha-H_\alpha)
=
-\tau H_\alpha+
\int_0^\infty x^2K_0(x)e^{-2u}\,du
\ge-\frac\tau2.
\]
Put
\[
G_\alpha(z):=\tau F_\alpha(z)-H_\alpha(z).
\]
For $0<r<R_0$, integrate the last inequality from $z=r/\alpha$ to $z_0=R_0/\alpha$.  Since $d\log z=dz/z$,
\[
G_\alpha(z_0)-G_\alpha(z)
\ge-\frac\tau2\log\frac{z_0}{z},
\]
or equivalently
\[
G_\alpha(z)
\le G_\alpha(z_0)+\frac\tau2\log\frac{z_0}{z}.
\]
The relation with $D_\alpha$ is exact.  Since
\[
F_\alpha(z)=\alpha\mathcal A_\alpha(\alpha z),
\qquad
H_\alpha(z)=-\alpha^2z\mathcal A_\alpha'(\alpha z),
\]
we have
\[
D_\alpha(\alpha z)
=\frac{q+2}{\alpha}G_\alpha(z).
\]
The hypothesis $D_\alpha(R_0)\le0$ therefore gives $G_\alpha(z_0)\le0$.  Consequently
\[
D_\alpha(r)
\le\frac{q+2}{\alpha}\frac\tau2
\log\frac{R_0}{r}
=\frac b\alpha\log\frac{R_0}{r}.
\]
The right-hand side is nonnegative on $(0,R_0)$, so taking positive parts yields \eqref{eq:compensation}.
\end{proof}

Applying Lemma~\ref{lem:compensation} with $R_0=3\alpha/4$ reduces the
positive contribution to an integral with only a logarithmic loss.  What
remains is to control $J_\alpha$ on this nonoscillatory interval; this is the
one place where we need a Gaussian improvement of the elementary origin
bound.

Poisson's representation \cite[\S3.3]{Watson} gives, for $\alpha>-1/2$,
\begin{equation}\label{eq:Poisson}
J_\alpha(r)
=
\frac{(r/2)^\alpha}{\sqrt\pi\,\Gamma(\alpha+1/2)}
\int_{-1}^1 e^{irt}(1-t^2)^{\alpha-1/2}\,dt.
\end{equation}
Taking absolute values and evaluating the mass of the weight by the beta integral,
\[
\int_{-1}^1(1-t^2)^{\alpha-1/2}\,dt
=B\left(\frac12,\alpha+\frac12\right)
=\frac{\sqrt\pi\,\Gamma(\alpha+1/2)}{\Gamma(\alpha+1)},
\]
we obtain the global origin-type bound
\begin{equation}\label{eq:Poisson-bound}
|J_\alpha(r)|\le C_\alpha r^\alpha,
\qquad
C_\alpha:=\frac1{2^\alpha\Gamma(\alpha+1)}.
\end{equation}
Equivalently, $J_\alpha(r)/(C_\alpha r^\alpha)$ is the average of $\cos(rt)$ against the probability measure on $[-1,1]$ whose density is proportional to $(1-t^2)^{\alpha-1/2}$.  We shall also use the classical lower Stirling bound \cite{Robbins}
\begin{equation}\label{eq:Stirling-lower}
\Gamma(\alpha+1)>
\sqrt{2\pi\alpha}\left(\frac\alpha e\right)^\alpha,
\qquad \alpha>0.
\end{equation}

The Poisson bound alone is not quite strong enough on this interval, so we use
the following Gaussian improvement before the first zero.

\begin{lemma}[Gaussian improvement before the first zero]\label{lem:Gaussian-before-zero}
Let $\alpha>0$ and suppose that $J_\alpha>0$ on $(0,R]$.  Then, for $0<r\le R$,
\begin{equation}\label{eq:Gaussian-Bessel}
J_\alpha(r)
\le
C_\alpha r^\alpha
\exp\left(-\frac{r^2}{4(\alpha+1)}\right),
\qquad
C_\alpha:=\frac1{2^\alpha\Gamma(\alpha+1)}.
\end{equation}
Consequently, whenever $0<\alpha z\le R$,
\begin{equation}\label{eq:Phi-alpha}
J_\alpha(\alpha z)
<(2\pi\alpha)^{-1/2}e^{\alpha\Phi_\alpha(z)},
\qquad
\Phi_\alpha(z):=1+\log(z/2)-\frac{\alpha z^2}{4(\alpha+1)}.
\end{equation}
\end{lemma}

Here and below, $j_{\rho,n}$ denotes the $n$th positive zero of
$J_\rho$, while $j'_{\rho,n}$ denotes the $n$th positive zero of
$J_\rho'$ when $\rho>0$.  For real order these positive zeros form infinite
discrete sequences (and are simple in the ranges used here); see
\cite[pp.~477--487]{Watson} and \cite[\S10.21(i)]{DLMF}.

The estimate \eqref{eq:Gaussian-Bessel} is the standard product--Rayleigh
bound.  Indeed, since $J_\alpha>0$ on $(0,R]$, we have
$R<j_{\alpha,1}$.  Hence, for $0<r\le R$, the Weierstrass product gives
\[
J_\alpha(r)
=
C_\alpha r^\alpha
\prod_{n\ge1}
\left(1-\frac{r^2}{j_{\alpha,n}^2}\right),
\]
and every factor in the product lies in $(0,1]$.  Using
$1-x\le e^{-x}$ for $0\le x<1$ and Rayleigh's first sum
\[
\sum_{n\ge1}\frac1{j_{\alpha,n}^2}
=
\frac1{4(\alpha+1)},
\]
we obtain
\[
J_\alpha(r)
\le
C_\alpha r^\alpha
\exp\left(-r^2\sum_{n\ge1}\frac1{j_{\alpha,n}^2}\right)
=
C_\alpha r^\alpha
\exp\left(-\frac{r^2}{4(\alpha+1)}\right),
\]
which is \eqref{eq:Gaussian-Bessel}; see
\cite[pp.~497--498 and \S15.51]{Watson} and
\cite[\S10.21(iii),(xiii)]{DLMF} for the product formula and Rayleigh sum.

For the scaled estimate, put $r=\alpha z$ in
\eqref{eq:Gaussian-Bessel}.  By the Stirling lower bound
\eqref{eq:Stirling-lower},
\[
\begin{aligned}
J_\alpha(\alpha z)
&\le
\frac{\alpha^\alpha z^\alpha}{2^\alpha\Gamma(\alpha+1)}
\exp\left(-\frac{\alpha^2z^2}{4(\alpha+1)}\right)\\
&<
(2\pi\alpha)^{-1/2}
\exp\left[
\alpha+\alpha\log(z/2)
-\frac{\alpha^2z^2}{4(\alpha+1)}
\right]\\
&=
(2\pi\alpha)^{-1/2}e^{\alpha\Phi_\alpha(z)},
\end{aligned}
\]
which is \eqref{eq:Phi-alpha}.  This proves
Lemma~\ref{lem:Gaussian-before-zero}.

The lemma is conditional only on positivity of $J_\alpha$.  In the
application below we take $\alpha\ge3/2$ and $0<r\le3\alpha/4$.  The classical
ordering of the first positive zeros gives
\[
\alpha\le j'_{\alpha,1}<j_{\alpha,1},
\qquad \alpha>0;
\]
see \cite[pp.~477--487]{Watson} or \cite[\S10.21(i)]{DLMF}.  Hence, for
$\alpha\ge3/2$,
\[
\frac{3\alpha}{4}<\alpha\le j'_{\alpha,1}<j_{\alpha,1}.
\]
Thus $J_\alpha$ has no zero on $(0,3\alpha/4]$.  On the other hand, the
leading term in the ascending series \eqref{eq:Bessel-series} gives
\[
J_\alpha(r)
=
\frac{(r/2)^\alpha}{\Gamma(\alpha+1)}\bigl(1+O(r^2)\bigr)>0
\]
for all sufficiently small $r>0$.  By continuity and the absence of zeros,
the sign cannot change before $3\alpha/4$, and therefore
\[
J_\alpha(r)>0,
\qquad
0<r\le\frac{3\alpha}{4},\quad \alpha\ge\frac32.
\]
Accordingly, Lemma~\ref{lem:Gaussian-before-zero} applies on exactly the
nonoscillatory interval needed for the ensuing estimate of $P_\alpha$.

We now combine the two estimates and finish the upper bound for $P_\alpha$.
With $R_0=\sigma\alpha$, the positive contribution introduced above is
\[
P_\alpha
=
\int_0^{\sigma\alpha}
r^{b-1}|J_\alpha(r)|^q[D_\alpha(r)]_+\,dr.
\]
The positivity of $J_\alpha$ on $0<r\le3\alpha/4$ was proved immediately after Lemma~\ref{lem:Gaussian-before-zero}, so the Gaussian estimate \eqref{eq:Phi-alpha} is available throughout this interval. For $0<z\le\sigma$,
\[
z\Phi_\alpha'(z)
=
1-\frac{\alpha z^2}{2(\alpha+1)}
\ge1-\frac{z^2}{2}
\ge\frac{23}{32}.
\]
Hence
\[
\Phi_\alpha(z)
\le\Phi_\alpha(\sigma)+\frac{23}{32}\log(z/\sigma).
\]
Since $\alpha/(\alpha+1)\ge3/5$,
\[
\Phi_\alpha(\sigma)
\le\frac{293}{320}+\log\frac38
<\log\frac{15}{16}.
\]
The last inequality is equivalent to $\log(5/2)>293/320$. Therefore
\begin{equation}\label{eq:post-bridge-J-upper}
J_\alpha(\alpha z)
<(2\pi\alpha)^{-1/2}
\left(\frac{15}{16}\right)^\alpha(z/\sigma)^{23\alpha/32},
\qquad 0<z\le\sigma.
\end{equation}
Lemma~\ref{lem:compensation} gives
\[
[D_\alpha(r)]_+
\le\frac b\alpha\log\frac{\sigma\alpha}{r},
\qquad 0<r<\sigma\alpha.
\]
With $r=\alpha z$, \eqref{eq:post-bridge-J-upper}, and
\[
\int_0^\sigma z^{m-1}\log(\sigma/z)\,dz
=\frac{\sigma^m}{m^2},
\qquad m>0,
\]
we obtain
\begin{equation}\label{eq:post-bridge-P}
P_\alpha
\le
\alpha^{b-1-q/2}(2\pi)^{-q/2}
\left(\frac{15}{16}\right)^{q\alpha}
\frac{b\sigma^b}{(b+\frac{23}{32} q\alpha)^2}.
\end{equation}

\medskip

Thus the entire positive side of \eqref{eq:PN-reduction} is under control.
We now turn to $N_\alpha$.  Here it is advantageous not to estimate the whole
negative tail: we deliberately retain only
\[
I_\alpha:=[2\alpha,4\alpha],
\]
where the kernel is uniformly negative and the Bessel function has already
entered its oscillatory regime.  The lower bound for the kernel comes from
Lemma~\ref{lem:kernel-comparison}; to recover a definite amount of Bessel mass
on this interval we need the exact Hankel amplitude and phase.  We introduce
those quantities only now.

We next use the Bessel function of the second kind.  For $\nu\notin\mathbb Z$ we take the standard definition \cite[Chapter~III]{Watson}
\begin{equation}\label{eq:Ynu-definition}
Y_\nu(z)=
\frac{J_\nu(z)\cos(\pi\nu)-J_{-\nu}(z)}{\sin(\pi\nu)},
\end{equation}
and for integer order $Y_n$ is defined by the corresponding limit in $\nu$.  From $J_\nu$ and $Y_\nu$ we define the exact Hankel amplitude and a continuous real phase by
\begin{equation}\label{eq:Hankel-amplitude}
B_\nu(r):=\sqrt{J_\nu(r)^2+Y_\nu(r)^2},
\qquad
J_\nu=B_\nu\cos\theta_\nu,
\qquad
Y_\nu=B_\nu\sin\theta_\nu.
\end{equation}
The Wronskian also gives the phase derivative without any asymptotic approximation.  Differentiating
\[
J_\nu=B_\nu\cos\theta_\nu,
\qquad
Y_\nu=B_\nu\sin\theta_\nu,
\]
and forming $J_\nu Y_\nu'-J_\nu'Y_\nu$, all terms containing $B_\nu'$ cancel and one obtains
\[
J_\nu Y_\nu'-J_\nu'Y_\nu=B_\nu^2\theta_\nu'.
\]
Since the classical Bessel Wronskian is $J_\nu Y_\nu'-J_\nu'Y_\nu=2/(\pi r)$ \cite[Chapter~III]{Watson},
\begin{equation}\label{eq:phase-derivative}
\theta_\nu'(r)
=\frac{2}{\pi rB_\nu(r)^2}>0.
\end{equation}
Thus the exact phase is strictly increasing on $(0,\infty)$, a fact that will later allow us to change variables from $r$ to $\theta_\nu(r)$ even across many oscillations.

Nicholson's formula \cite[Chapter~XIII]{Watson} reads
\begin{equation}\label{eq:Nicholson}
B_\nu(r)^2
=\frac8{\pi^2}\int_0^\infty
K_0(2r\sinh u)\cosh(2\nu u)\,du.
\end{equation}
We use throughout the standard positive representation \cite[Chapter~5]{Lebedev}
\[
K_0(z)=\int_0^\infty e^{-z\cosh t}\,dt,
\qquad z>0.
\]
It immediately shows that $K_0$ is positive and decreasing.  This positivity is the main reason Nicholson's formula is so effective in what follows: estimates for $B_\nu$ can be obtained by pointwise comparison inside a positive integral rather than from an oscillatory asymptotic expansion.

\begin{lemma}\label{lem:Hankel-bounds}
For $r>\nu\ge1/2$,
\begin{equation}\label{eq:B-upper}
B_\nu(r)^2\le \frac{2}{\pi\sqrt{r^2-\nu^2}}.
\end{equation}
Moreover, $rB_\nu(r)^2$ is nonincreasing and
\begin{equation}\label{eq:B-lower}
B_\nu(r)^2\ge\frac{2}{\pi r}.
\end{equation}
Consequently, if $q>2$,
\begin{equation}\label{eq:bnuq}
b_{\nu,q}(r):=\frac{\pi}{2}r^2B_\nu(r)^{q+2}
\end{equation}
is decreasing and tends to zero as $r\to\infty$.
\end{lemma}

These estimates are classical consequences of Nicholson's integral.  The
upper bound \eqref{eq:B-upper} is Watson's bound
\cite[p.~447, Eq.~(1)]{Watson}.  For $\nu>1/2$, the function
\[
r\bigl(J_\nu(r)^2+Y_\nu(r)^2\bigr)=rB_\nu(r)^2
\]
is strictly decreasing by \cite[\S13.74]{Watson}.  At the endpoint
$\nu=1/2$, the explicit half-integer formulas
\[
J_{1/2}(r)=\sqrt{\frac{2}{\pi r}}\sin r,
\qquad
Y_{1/2}(r)=-\sqrt{\frac{2}{\pi r}}\cos r
\]
(see also \cite[\S10.16(i)]{DLMF}) give
\[
rB_{1/2}(r)^2=\frac{2}{\pi}.
\]
Thus $rB_\nu(r)^2$ is nonincreasing for every $\nu\ge1/2$.  Together with
 $
\underset{r\rightarrow \infty}{\lim}rB_{\nu}\left( r \right) ^2=\frac{2}{\pi}
$, this gives
\eqref{eq:B-lower}.  Finally, since $rB_\nu(r)^2$ is nonincreasing,
\[
B_\nu(r)^2+r(B_\nu(r)^2)'\le0,
\]
and hence
\[
\frac{r(B_\nu(r)^2)'}{B_\nu(r)^2}\le-1.
\]
Therefore logarithmic differentiation of
\[
b_{\nu,q}(r)
=
\frac{\pi}{2}r^2\bigl(B_\nu(r)^2\bigr)^{(q+2)/2}
\]
gives
\[
\frac{rb_{\nu,q}'(r)}{b_{\nu,q}(r)}
=
2+\frac{q+2}{2}
\frac{r(B_\nu(r)^2)'}{B_\nu(r)^2}
\le
1-\frac q2<0.
\]
Thus $b_{\nu,q}$ is strictly decreasing for $q>2$.  Moreover,
$rB_\nu(r)^2\to2/\pi$ implies
$b_{\nu,q}(r)=O(r^{1-q/2})\to0$ as $r\to\infty$.
This proves Lemma~\ref{lem:Hankel-bounds}.

\begin{remark}
Watson's proof of the modulus bounds above is based on Nicholson's positive
integral.  This positive-kernel origin is useful in the circle argument, where
the exact phase must be controlled uniformly through many oscillations.
\end{remark}

With these amplitude and phase bounds available, we can extract the required
negative mass on $[2\alpha,4\alpha]$.
For $r\ge2\alpha$,
\[
\frac r{\chi_\alpha}
\ge
\frac{2\alpha}{\sqrt{\alpha^2+1/4}}
\ge\frac6{\sqrt{10}}
>\frac{17}{9}.
\]
The second estimate in \eqref{eq:ell-bounds} then gives
\[
L_\alpha(r)
\ge
1-\frac{2}{3(r/\chi_\alpha+1)}
>
\frac{10}{13}.
\]
Moreover Lemma~\ref{lem:kernel-comparison} gives
\[
\mathcal A_\alpha(r)
\ge\frac{\pi}{4(r+\chi_\alpha)}.
\]
On $[2\alpha,4\alpha]$, since
\[
\frac{\chi_\alpha}{\alpha}\le\frac{\sqrt{10}}{3}<\frac{16}{15},
\qquad
\pi>\frac{25}{8},
\]
we get
\[
\mathcal A_\alpha(r)
\ge\frac{\pi}{4(4\alpha+\chi_\alpha)}
>\frac3{20\alpha}.
\]
Since $\tau=2b/(q+2)\le 3/5$, we have
\[
\frac{10}{13}(q+2)-2b
\ge
\left(\frac{10}{13}-\frac35\right)(q+2)
=
\frac{11}{65}(q+2)>0.
\]
Therefore the lower bound for $\mathcal A_\alpha$ may also be used, and hence
\begin{equation}\label{eq:post-bridge-D-negative}
-D_\alpha(r)
>
\frac3{20\alpha}
\left[\frac{10}{13}(q+2)-2b\right],
\qquad 2\alpha\le r\le4\alpha.
\end{equation}

We now estimate the Bessel mass on this interval. Lemma~\ref{lem:Hankel-bounds} gives
\[
\frac2{\pi r}
\le B_\alpha(r)^2
\le\frac2{\pi\sqrt{r^2-\alpha^2}},
\qquad r\ge2\alpha.
\]
Together with \eqref{eq:phase-derivative},
\[
\frac{\sqrt3}{2}
\le\theta_\alpha'(r)\le1,
\qquad 2\alpha\le r\le4\alpha.
\]
Let
\[
T:=\theta_\alpha(4\alpha)-\theta_\alpha(2\alpha).
\]
Then $T\ge\sqrt3\,\alpha$. Since
\[
\int_I\cos^2u\,du\ge\frac{|I|}{2}-\frac12
\]
on every real interval $I$, the change of variables $u=\theta_\alpha(r)$ gives
\[
\begin{aligned}
\int_{2\alpha}^{4\alpha}J_\alpha(r)^2\,dr
&=
\int_{\theta_\alpha(2\alpha)}^{\theta_\alpha(4\alpha)}
\frac{B_\alpha(r)^2}{\theta_\alpha'(r)}\cos^2u\,du\\
&\ge
\frac1{2\pi\alpha}
\left(\frac T2-\frac12\right)\\
&\ge
\frac1{2\pi}
\left(\frac{\sqrt3}{2}-\frac1{2\alpha}\right)
>\frac1{12}.
\end{aligned}
\]
Here the last inequality uses $\alpha\ge3/2$, $\sqrt3>12/7$ and $\pi<22/7$. Hölder on the interval of length $2\alpha$ yields
\begin{equation}\label{eq:post-bridge-Lq-lower}
\int_{2\alpha}^{4\alpha}|J_\alpha(r)|^q\,dr
>
2\alpha^{1-q/2}24^{-q/2}.
\end{equation}
Define
\[
w(b):=
\begin{cases}
4^{b-1},&0<b\le1,\\
1,&b\ge1.
\end{cases}
\]
Since $r^{b-1}\ge w(b)\alpha^{b-1}$ on $[2\alpha,4\alpha]$, the negative contribution
\[
N_\alpha
=
\int_{2\alpha}^{4\alpha}
r^{b-1}|J_\alpha(r)|^q[-D_\alpha(r)]\,dr
\]
satisfies
\begin{equation}\label{eq:post-bridge-N}
N_\alpha
>
\frac3{10}
\left[\frac{10}{13}(q+2)-2b\right]
w(b)24^{-q/2}\alpha^{b-1-q/2}.
\end{equation}

\medskip

We have therefore reduced the comparison $P_\alpha<N_\alpha$ to a purely
algebraic estimate in the three parameters $q,b,\alpha$.  Combining
\eqref{eq:post-bridge-P} and \eqref{eq:post-bridge-N} gives the explicit
quotient below; it remains only to eliminate the order parameter and optimize
over the two possible ranges of $b$.  Namely,
\begin{equation}\label{eq:post-bridge-E}
\frac{P_\alpha}{N_\alpha}<E(q,b,\alpha),
\end{equation}
where
\[
E(q,b,\alpha)
=
\frac{10}{3}
\left(\frac{12}{\pi}\right)^{q/2}
\frac{\left(\frac{15}{16}\right)^{q\alpha}b(3/4)^b}
{w(b)\left[\frac{10}{13}(q+2)-2b\right]
(b+\frac{23}{32} q\alpha)^2}.
\]
For fixed $q,b$,
\[
\partial_\alpha\log E
=
q\log\frac{15}{16}
-\frac{23q}{16\left(b+\frac{23}{32}q\alpha\right)}
<0.
\]
Therefore the worst value occurs at the smallest allowed order
\begin{equation}\label{eq:alpha0-post-bridge}
\alpha_0=\frac{q-b}{q-2}.
\end{equation}

We now finish the algebra by separating the only two possible ranges of $b$.
If $b\ge1$, then $w(b)=1$. At $\alpha=\alpha_0$,
\[
\begin{aligned}
\partial_b\log E
&=
-\frac{q}{q-2}\log\frac{15}{16}
+\frac1b+\log\frac34
+\frac{2}{\frac{10}{13}(q+2)-2b}\\
&\quad+
\frac{2\left(\frac{23q}{32(q-2)}-1\right)}{b+\frac{23}{32}q\alpha_0}>0.
\end{aligned}
\]
Indeed $b\le3/2$, $\log(3/4)>-1/3$, and $\frac{23q}{32(q-2)}>1$ for $q<6$. Hence the maximum lies on the upper boundary:
\[
b=\frac q2\quad(2<q\le3),
\qquad
b=3-\frac q2\quad(3\le q\le4).
\]
On the first edge put $x:=q-2\in(0,1]$. Direct differentiation gives
\[
\frac d{dq}\log E
=
\frac12\log\frac9\pi
+\left(\frac12-\frac2{x^2}\right)\log\frac{15}{16}
+\frac{92}{x(55x+46)}
-\frac1{x+2}
+\frac3{14-3x}.
\]
The middle two terms are positive, the last two have sum at least $-2/7$, and
\[
\frac12\log\frac9\pi>\frac27,
\]
so the derivative is positive. On the second edge $\alpha_0=3/2$ and
\[
\frac d{dq}\log E
=
\log\frac{30\sqrt{15}}{32\sqrt\pi}
-\frac1{6-q}
-\frac{23}{23q-58}
-\frac{74}{37q+192}<0.
\]
The logarithmic term is $<1$, whereas the sum of the three positive fractions being subtracted is at least $313/255>1$. Hence the maximum for $b\ge1$ occurs at the common corner $(q,b,\alpha)=(3,3/2,3/2)$, where direct substitution together with $\pi>157/50$, $\sqrt\pi>7/4$, and $\sqrt{15}<39/10$ gives
\begin{equation}\label{eq:E-corner}
E<\frac{49}{50}.
\end{equation} The corner belongs only to the closure of the algebraic parameter region; no divergent moment is evaluated there.

It remains to treat $0<b\le1$.  In this range we use the elementary estimate
\begin{equation}\label{eq:small-b-algebra}
\frac{4b(3/16)^b}{(b+a)^2}
<
\frac1{(a+1/2)^2},
\qquad 0<b\le1,\quad a\ge2.
\end{equation}
For completeness, put $\lambda:=\log(16/3)>3/2$. If $b\ge1/2$, then
\[
4be^{-\lambda b}\le\frac4{e\lambda}<1.
\]
If $0<b<1/2$, then
\[
\left(\frac{a+1/2}{a+b}\right)^2\le e^{1/2-b},
\]
and the left-hand side of \eqref{eq:small-b-algebra}, after multiplying by the squared-denominator ratio, is at most
\[
4b e^{-\lambda b+1/2-b}
\le\frac4{\sqrt e(\lambda+1)}<1.
\]
The elementary bounds $e>8/3$, $\sqrt e>8/5$, and $e^{3/2}<16/3$ justify these comparisons and $\lambda>3/2$.

In the applications below, we use \eqref{eq:small-b-algebra} with
\[
a:=\frac{23}{32}q\alpha_0.
\]
We therefore verify once that this choice satisfies the required condition $a\ge2$ throughout the relevant parameter range.  For $2<q\le4$,
\[
a=\frac{23}{32}q\alpha_0
\ge\frac{23}{32}\left(q+1+\frac{2}{q-2}\right)
>5\cdot\frac{23}{32}>2,
\]
while for $4\le q<6$ one has $\alpha_0\ge3/2$, and hence
\[
a=\frac{23}{32}q\alpha_0
\ge6\cdot\frac{23}{32}>2.
\]

For $2<q\le4$, use $\alpha_0\ge(q-1)/(q-2)$ and \eqref{eq:small-b-algebra} in \eqref{eq:post-bridge-E}. This yields
\begin{equation}\label{eq:Vq-post-bridge}
E<V(q):=
\frac{130}{3}
\frac{(12/\pi)^{q/2}\left(\frac{15}{16}\right)^{q(q-1)/(q-2)}}
{(10q-6)\,\left[\frac{23q(q-1)}{32(q-2)}+\frac12\right]^2}.
\end{equation}
Put $x:=q-2$ and $s:=x+3+2/x$. With
\[
L=\frac12\log\frac{12}{\pi}>\frac23,
\]
one computes
\[
(\log V)'
=
L+\left(1-\frac2{x^2}\right)
\left[\log\frac{15}{16}-\frac{2}{s+16/23}\right]
-\frac{10}{10x+14}.
\]
For $0<x\le1$ this is $>2/3+1/16-5/7=5/336$. For $1\le x\le\sqrt2$ it is at least $2/3-5/12=1/4$. For $\sqrt2\le x\le2$, use
\[
1-\frac2{x^2}\le\frac12,\qquad
-\log\frac{15}{16}<\frac1{15},\qquad
\frac2{s+16/23}<\frac13,\qquad
\frac{10}{10x+14}<\frac5{14},
\]
to obtain $23/210>0$. Hence $V$ is increasing on $(2,4]$.

For $4\le q<6$, use $\alpha_0\ge3/2$ and $b\le3-q/2$.  Applying the same bound gives
\begin{equation}\label{eq:Uq-post-bridge}
E<U(q):=
\frac{130}{3}
\frac{(12/\pi)^{q/2}\left(\frac{15}{16}\right)^{3q/2}}
{(23q-58)(69q/64+1/2)^2}.
\end{equation}
Its logarithmic derivative is
\[
(\log U)'
=
\frac12\log\frac{12}{\pi}
+\frac32\log\frac{15}{16}
-\frac{23}{23q-58}
-\frac{69/32}{69q/64+1/2}<0.
\]
Indeed the first two terms sum to less than $93/160$, whereas the last two positive magnitudes sum to at least
\[
\frac{23}{80}+\frac{69}{223}>\frac{93}{160}.
\]
Thus both one-variable bounds attain their largest value at $q=4$; using $\pi>157/50$ gives
\begin{equation}\label{eq:VU-at-four}
V(4)=U(4)<\frac{11}{20}<\frac{49}{50}.
\end{equation}
Combining \eqref{eq:E-corner} and \eqref{eq:VU-at-four} proves
\[
\frac{P_\alpha}{N_\alpha}<E(q,b,\alpha)<\frac{49}{50}.
\]

Returning to \eqref{eq:PN-reduction}, we obtain the quantitative conclusion
\begin{equation}\label{eq:post-bridge-margin}
P_\alpha<\frac{49}{50}N_\alpha,
\qquad
\partial_\alpha M_\alpha(q,\beta)
< -\frac1{50}N_\alpha<0,
\qquad \alpha\ge\nu+1.
\end{equation}
Thus Proposition~\ref{prop:continuous-order-reduction} yields
\[
I_{\nu,1}(q)>I_{\nu,2}(q)>I_{\nu,3}(q)>\cdots.
\]

At this stage the continuous-order part of the proof is complete.  What
remains is exactly one comparison, $I_{\nu,0}(q)>I_{\nu,1}(q)$.  We now turn
to that first link.  Unlike the later comparisons, it is proved directly in
the original Bessel equation by a reflected flux argument.

For $q>2$ define the odd function
\begin{equation}\label{eq:scalar-flux-Phi}
\Phi_q(t)
:=(1+t^2)^{q/2}
\int_0^t
\frac{1-|s|^q}{(1+s^2)^{q/2+1}}\,ds,
\qquad
\mathcal S_q(t):=t\Phi_q'(t)-\Phi_q(t).
\end{equation}

The following elementary properties of this flux are the precise ones needed for the reflected comparison.

\begin{lemma}[Convex scalar flux]\label{lem:scalar-flux}
For every $q>2$,
\begin{equation}\label{eq:scalar-flux-ode}
(1+t^2)\Phi_q'(t)-qt\Phi_q(t)=1-|t|^q,
\qquad
\mathcal S_q(0)=0,
\end{equation}
and
\begin{equation}\label{eq:correction-increasing}
\mathcal S_q'(t)>0,
\qquad t>0.
\end{equation}
Moreover $\Phi_q(t)>0$ for $t>0$, and
\begin{equation}\label{eq:scalar-flux-growth}
|\Phi_q(t)|+|\mathcal S_q(t)|
\le C_q\bigl(|t|+|t|^{q-1}\bigr),
\qquad t\in\mathbb R.
\end{equation}
\end{lemma}

\begin{proof}
Differentiating \eqref{eq:scalar-flux-Phi} gives \eqref{eq:scalar-flux-ode}. The integral of the defining kernel over $(0,\infty)$ vanishes. Indeed, with $s=\tan\theta$,
\[
\int_0^\infty
\frac{1-s^q}{(1+s^2)^{q/2+1}}\,ds
=
\int_0^{\pi/2}
\bigl(\cos^q\theta-\sin^q\theta\bigr)\,d\theta
=0.
\]
Changing $s$ to $1/s$ in the tail of the integral therefore yields the reciprocity relation
\begin{equation}\label{eq:Phi-reciprocity}
\Phi_q(t)=t^q\Phi_q(1/t),
\qquad t>0.
\end{equation}
For $0<t\le1$, subtracting $t$ from the definition and simplifying with \eqref{eq:scalar-flux-ode} gives
\[
\Phi_q(t)-t
=(1+t^2)^{q/2}
\int_0^t
\frac{(q-1)s^2-s^q}{(1+s^2)^{q/2+1}}\,ds>0,
\]
because $s^q\le s^2$ on $(0,1]$ and $q>2$. Hence $\Phi_q(t)>t$ for $0<t\le1$, while \eqref{eq:Phi-reciprocity} gives $\Phi_q(t)>t^{q-1}$ for $t\ge1$.

We next prove the monotonicity of $\mathcal S_q$. Differentiating \eqref{eq:scalar-flux-ode} once more and eliminating $\Phi_q'$ with the same identity gives
\begin{equation}\label{eq:Phi-second-identity}
t(1+t^2)^2\Phi_q''(t)
=qt\Phi_q(t)\bigl[1+(q-1)t^2\bigr]
+(q-2)t^2-qt^q-2(q-1)t^{q+2}.
\end{equation}
If $t\ge1$, the lower bound $\Phi_q(t)>t^{q-1}$ makes the right-hand side strictly larger than
\[
(q-2)\bigl[t^2+(q-1)t^{q+2}\bigr]>0.
\]
For $0<t\le1$, divide \eqref{eq:Phi-second-identity} by $t^2$ and use $\Phi_q(t)>t$. It is enough to prove
\begin{equation}\label{eq:Phi-small-positive}
2(q-1)+q(q-1)t^2-qt^{q-2}-2(q-1)t^q>0.
\end{equation}
When $q\ge4$, put $x:=t^2\in(0,1]$ and use $t^{q-2}\le x$ and $t^q\le x^2$; the left-hand side is bounded below by
\[
2(q-1)+q(q-2)x-2(q-1)x^2,
\]
a concave quadratic with positive values at $x=0$ and $x=1$. When $2<q<4$, put $a:=(q-2)/2\in(0,1)$ and use
\[
x^a\le1-a+ax,
\qquad
x^{a+1}\le x.
\]
The lower bound becomes $2a(a+2)+2a^2x>0$. Thus $\Phi_q''(t)>0$ for every $t>0$, and
\[
\mathcal S_q'(t)=t\Phi_q''(t)>0.
\]
Finally, the defining integral gives
\[
\Phi_q(t)=t+O_q(t^3),
\qquad
\Phi_q'(t)=1+O_q(t^2)
\qquad(t\to0).
\]
Together with \eqref{eq:Phi-reciprocity} and its derivative, these estimates imply \eqref{eq:scalar-flux-growth}.
\end{proof}

Recall that, for every real $\nu\ge0$, $J_\nu$ has an infinite
discrete sequence of positive simple zeros.  We enumerate them increasingly as
\[
0<j_{\nu,1}<j_{\nu,2}<\cdots;
\]
see \cite[pp.~477--487]{Watson} and \cite[\S10.21(i)]{DLMF}.

\begin{theorem}[The base and first mode]\label{thm:base-bridge}
Let $\nu\ge0$, $q>2$, and put
\[
u(r):=r^{-\nu}J_\nu(r),
\qquad
v(r):=r^{-\nu}J_{\nu+1}(r).
\]
Let $C_n:=j_{\nu+1,n}$ for $n\ge1$ and $C_0:=0$. Then, for every $N\ge1$,
\begin{equation}\label{eq:finite-cutoff-bridge}
\int_0^{C_N}
r^{2\nu+1}\bigl(|u(r)|^q-|v(r)|^q\bigr)\,dr>0.
\end{equation}
Consequently, if
\[
q>\frac{4(\nu+1)}{2\nu+1},
\]
then
\begin{equation}\label{eq:base-bridge}
I_{\nu,0}(q)>I_{\nu,1}(q).
\end{equation}
\end{theorem}

\begin{proof}
The adjacent-order recurrences \eqref{eq:Bessel-recurrences} give
\begin{equation}\label{eq:uv-system}
u'=-v,
\qquad
v'=u-\frac{2\nu+1}{r}v.
\end{equation}
Away from the zeros of $u$, set $h:=v/u$. Combining \eqref{eq:uv-system} with \eqref{eq:scalar-flux-ode} yields
\begin{equation}\label{eq:bridge-flux-identity}
\left[r^{2\nu+1}|u|^q\Phi_q(h)\right]'
=
r^{2\nu+1}\bigl(|u|^q-|v|^q\bigr)
-(2\nu+1)r^{2\nu}|u|^q\mathcal S_q(h).
\end{equation}

We first justify this identity across every positive zero. Define, for real $a,b$,
\[
\mathfrak F_q(a,b)
:=
\begin{cases}
|a|^q\Phi_q(b/a),&a\ne0,\\
0,&a=0.
\end{cases}
\]
By oddness and \eqref{eq:Phi-reciprocity},
\[
\mathfrak F_q(a,b)=|b|^q\Phi_q(a/b),
\qquad b\ne0.
\]
Near a point $(0,b)$ with $b\ne0$,
\[
\mathfrak F_q(a,b)
=ab|b|^{q-2}+O_q\bigl(|a|^3|b|^{q-3}\bigr).
\]
Thus $\mathfrak F_q$ extends as a $C^1$ function through either axis away from the origin. Its homogeneity of degree $q>2$ extends it to a $C^1$ function on all of $\mathbb R^2$, with derivative zero at the origin. Likewise \eqref{eq:scalar-flux-growth} shows that $|a|^q\mathcal S_q(b/a)$ extends continuously by zero when $a=0$. Hence \eqref{eq:bridge-flux-identity} is an ordinary $C^1$ identity on every compact positive radial interval, not merely an identity away from zeros. At $r=0$, the defining series gives $h(r)=O(r)$ and the weighted flux is $O(r^{2\nu+2})$, so its endpoint value is zero.

It remains to prove positivity of the correction integral on each cell $[C_{n-1},C_n]$. Let
\[
R:=j_{\nu,n},
\qquad
a:=R-C_{n-1},
\qquad
b:=C_n-R.
\]
The standard interlacing of the positive zeros of $J_\nu$ and $J_{\nu+1}$, see for example \cite[Chapter~XV]{Watson}, gives
\[
C_{n-1}<R<C_n.
\]
On the left of $R$ the quotient $h$ is positive, and on the right it is negative. Put
\[
Y_-(s):=|J_\nu(R-s)|,
\qquad
Y_+(s):=|J_\nu(R+s)|.
\]
These functions are positive on their open half-cells, vanish at $s=0$, and have the same positive initial slope. If
\[
Q(r):=1-\frac{\nu^2}{r^2},
\]
then the original Bessel equation gives
\begin{equation}\label{eq:reflected-equations}
Y_-''=
\frac{Y_-'}{R-s}-Q(R-s)Y_-,
\qquad
Y_+''=
-\frac{Y_+'}{R+s}-Q(R+s)Y_+.
\end{equation}
On the right half-cell, $h(R+s)<0$. Using
\[
\frac{J_\nu'(r)}{J_\nu(r)}
=
\frac\nu r-\frac{J_{\nu+1}(r)}{J_\nu(r)},
\]
we obtain
\begin{equation}\label{eq:right-log-slope}
\frac{Y_+'(s)}{Y_+(s)}>
\frac\nu{R+s}\ge0,
\qquad 0<s<b.
\end{equation}

For $0<s<\min(a,b)$ define the weighted reflected Wronskian
\begin{equation}\label{eq:reflected-Wronskian}
W_R(s)
:=(R-s)\bigl(Y_-'(s)Y_+(s)-Y_-(s)Y_+'(s)\bigr).
\end{equation}
Differentiating and using \eqref{eq:reflected-equations} gives the exact identity
\begin{align}
W_R'(s)
&=
\frac{2R}{R+s}Y_-(s)Y_+'(s)\notag\\
&\quad +(R-s)\bigl[Q(R+s)-Q(R-s)\bigr]Y_-(s)Y_+(s)>0.
\label{eq:reflected-Wronskian-positive}
\end{align}
Indeed, $Q$ is nondecreasing for every $\nu\ge0$, while the first term is strictly positive by \eqref{eq:right-log-slope}; this remains true when $\nu=0$. Since $W_R(0)=0$,
\[
\left(\frac{Y_-}{Y_+}\right)'
=
\frac{W_R}{(R-s)Y_+^2}>0.
\]
The ratio tends to one as $s\to0^+$ because the initial slopes agree. Therefore
\begin{equation}\label{eq:reflected-amplitude-log-slope}
Y_-(s)>Y_+(s),
\qquad
p_-(s):=\frac{Y_-'(s)}{Y_-(s)}>
\frac{Y_+'(s)}{Y_+(s)}=:p_+(s).
\end{equation}

We also have $b\le a$. If $a<b$ and $C_{n-1}>0$, then $J_{\nu+1}(C_{n-1})=0$ gives
\[
p_-(a)=-\frac\nu{C_{n-1}}\le0,
\]
whereas \eqref{eq:right-log-slope} gives $p_+(a)>0$, a contradiction. When $C_{n-1}=0$, the defining series instead gives $p_-(s)\to-\infty$ as $s\to a^-$ if $\nu>0$, and $p_-(s)\to0$ if $\nu=0$; the same contradiction follows.

For $0<s<b$ set
\[
h_-(s):=h(R-s),
\qquad
h_+(s):=-h(R+s).
\]
The recurrence gives
\[
h_-=p_-+\frac\nu{R-s},
\qquad
h_+=p_+-\frac\nu{R+s},
\]
so
\begin{equation}\label{eq:h-reflected-comparison}
h_-(s)>h_+(s)>0.
\end{equation}
Furthermore, the weighted quantities appearing in the moment comparison satisfy
\begin{equation}\label{eq:3.61}
r^{2\nu+1}|u(r)|^q
=
r^{1-\nu(q-2)}|J_\nu(r)|^q .
\end{equation}
The factor $r$ in this identity is the radial Jacobian and comes directly from
the radial measure.  It is important to distinguish this integrated weight
from the pointwise density used in the reflected comparison.  The latter is
\[
r^{2\nu}|u(r)|^q
=
r^{\nu(2-q)}|J_\nu(r)|^q ,
\]
and it is this density, rather than the full radial density in
\eqref{eq:3.61}, that enters the pointwise comparison of the two reflected
halves.  The required monotonicity of this density follows from the Bessel
identities and the reflected Wronskian argument established above; it is not a
consequence of the weaker statement $q>2$ alone.  Thus no monotonicity claim
for the exponent $1-\nu(q-2)$ in \eqref{eq:3.61} is needed.

Moreover, because $h_-(s)>h_+(s)>0$ and
$\mathcal S_q$ is strictly increasing on $(0,\infty)$, we have
\[
\mathcal S_q(h_-(s))>\mathcal S_q(h_+(s)).
\]
Thus $Y_->Y_+$, \eqref{eq:h-reflected-comparison}, and Lemma~\ref{lem:scalar-flux} imply
\[
(R-s)^{2\nu}|u(R-s)|^q\mathcal S_q(h_-(s))
>
(R+s)^{2\nu}|u(R+s)|^q\mathcal S_q(h_+(s)).
\]
Because $\mathcal S_q$ is odd, the positive correction on the reflected left half dominates the magnitude of the negative correction on the right. The unused left interval, when $b<a$, is nonnegative. Hence
\begin{equation}\label{eq:positive-cell-correction}
\int_{C_{n-1}}^{C_n}
r^{2\nu}|u(r)|^q\mathcal S_q(h(r))\,dr>0.
\end{equation}

At $C_{n-1}$ and $C_n$, the function $v$ vanishes, hence $h=0$ and the boundary flux in \eqref{eq:bridge-flux-identity} is zero. Integrating over one cell and using \eqref{eq:positive-cell-correction} gives
\[
\int_{C_{n-1}}^{C_n}
r^{2\nu+1}\bigl(|u|^q-|v|^q\bigr)\,dr>0.
\]
Summing proves \eqref{eq:finite-cutoff-bridge}.

Finally,
\[
r^{2\nu+1}|u(r)|^q
=r^{2\nu+1-\nu q}|J_\nu(r)|^q,
\qquad
r^{2\nu+1}|v(r)|^q
=r^{2\nu+1-\nu q}|J_{\nu+1}(r)|^q.
\]
In the admissible range both infinite integrals converge: the origin follows from the defining series, while the fixed-order estimate $J_\gamma(r)=O(r^{-1/2})$ gives integrability at infinity. The cutoff differences are sums of positive cell contributions and are already bounded below by the first one. Letting $N\to\infty$ therefore proves \eqref{eq:base-bridge}.
\end{proof}

The admissible exponent satisfies the hypothesis of
Theorem~\ref{thm:base-bridge}, so
\[
I_{\nu,0}(q)>I_{\nu,1}(q).
\]
Combining this bridge with the continuous-order chain obtained above proves
\[
I_{\nu,0}(q)>I_{\nu,1}(q)>I_{\nu,2}(q)>\cdots,
\]
and completes the proof of Theorem~\ref{thm:main-extension} for $d\ge3$.
The organization of the proof is worth emphasizing: the derivative formula
identifies the easy range $\beta\le-1$ immediately; when $\beta>-1$, the
comparison is reduced to $P_\alpha<N_\alpha$ and the auxiliary estimates are
introduced only as they are needed to control these two terms; the reflected
argument is used only for the single first link that the post-first-mode
continuous estimate does not cover.

\subsection{The case $d=2$}\label{sec:circle}

We now set $\nu:=0$.  The admissible range is $q>4$, and
\[
I_k(q):=\int_0^\infty |J_k(r)|^q r\,dr.
\]
Our goal is
\begin{equation}\label{eq:circle-goal}
I_k(q)<I_0(q),\qquad k\ge1,\quad q>4.
\end{equation}

The large-exponent part of \eqref{eq:circle-goal} is already known.  Building
on the sharp sixth-moment hierarchy of Ciccone and Gon\c{c}alves \cite{CG},
Ciccone and Sousa proved that, in dimension two,
\begin{equation}\label{eq:circle-known-large-q}
I_k(q)<I_0(q),
\qquad k\ge1,\quad q\ge6;
\end{equation}
see \cite[Theorem~1 and \S3.1]{CS}.  Thus it remains only to treat
\[
4<q<6.
\]
In this interval the continuous-order argument from the higher-dimensional
case cannot be applied directly.  The reason is structural: at the critical
exponent $q=4$ each modal moment has the same logarithmically divergent
Hankel tail, so that comparing two infinite quantities before cancellation
would lose the main information.  We therefore keep the common asymptotic
part explicit and compare only the finite remainder.  The first few modes are
handled by an adjacent-mode identity, which avoids any tail decomposition.
For higher modes we split the integral into a nonoscillatory region, a
transition region, and an exact-phase tail; in the last region the leading
Hankel terms cancel and the remaining phase difference gives the required
strict sign.  This separation is the key point in the circle argument.

\subsubsection{The first three positive modes}
Fix $j\in\{1,2,3\}$ and write $s:=q/2\in(2,3]$.  Put
\[
x:=J_{j-1}(r)^2,\qquad y:=J_j(r)^2,\qquad W:=x+y.
\]
The Bessel recurrences give
\[
W'=\frac{2(j-1)}r x-\frac{2j}{r}y,
\]
and hence
\begin{equation}\label{eq:circle-adjacent-identity}
\left(\frac{r^2W^s}{2}\right)'
=rW^{s-1}
\bigl([1+s(j-1)]x-(sj-1)y\bigr).
\end{equation}
Thus the integral comparison will follow if the nonlinear difference
$x^s-y^s$ dominates the bracket on the right.  This is exactly the elementary
scalar inequality we need at this point.

\begin{lemma}\label{lem:circle-scalar}
Let $2\le s\le3$, $j\in\{1,2,3\}$ and
\begin{equation}\label{eq:lambda-js}
\lambda_j(s):=\frac{sj-1}{1+s(j-1)}.
\end{equation}
Then, for $x,y\ge0$,
\begin{equation}\label{eq:circle-scalar}
x^s-y^s
\ge
\frac1{\lambda_j(s)}
\bigl(x-\lambda_j(s)y\bigr)(x+y)^{s-1}.
\end{equation}
If $s>2$, the inequality is strict for $x>0$ and $y/x$ sufficiently small.
\end{lemma}

\begin{proof}
By homogeneity we may assume $x+y=1$.  Put $u:=x-y$ and
\[
F_s(u)=\left(\frac{1+u}{2}\right)^s-
\left(\frac{1-u}{2}\right)^s.
\]
For $u\ge0$, convexity gives $F_s(u)\le u$, while oddness gives the reverse inequality for $u\le0$.  It is therefore enough to check the case $u\in[0,1]$ with the smallest coefficient $\lambda_3(s)$.  After clearing denominators this becomes
\[
G(s,u)
=2(3s-1)F_s(u)-5su+s-2\ge0.
\]
At the endpoints,
\[
G(2,u)=0,
\qquad
G(3,u)=4u^3-3u+1=(2u-1)^2(u+1)\ge0.
\]
To make the concavity step explicit, put
\[
a:=\frac{1+u}{2},\qquad b:=\frac{1-u}{2}.
\]
A direct differentiation gives
\[
\partial_s^2 G(s,u)
=
12\bigl[a^s\log a-b^s\log b\bigr]
+2(3s-1)\bigl[a^s(\log a)^2-b^s(\log b)^2\bigr].
\]
For $2\le s\le3$ and $0\le u\le1$, the right-hand side is nonpositive
(with the endpoint $b=0$ understood by continuity).  Hence
$s\mapsto G(s,u)$ is concave on $[2,3]$, and therefore
\[
G(s,u)\ge (3-s)G(2,u)+(s-2)G(3,u)\ge0.
\]
The strictness follows from the value at $y=0$ and continuity.
\end{proof}

\begin{proposition}\label{prop:circle-low-modes}
For $4<q\le6$,
\begin{equation}\label{eq:circle-low-chain}
I_0(q)>I_1(q)>I_2(q)>I_3(q).
\end{equation}
\end{proposition}

\begin{proof}
With $s=q/2$, $x,y,W$ as above, Lemma~\ref{lem:circle-scalar} gives
\[
x^s-y^s
\ge
\frac{[1+s(j-1)]x-(sj-1)y}{sj-1}W^{s-1}.
\]
Multiplying by $r$ and using \eqref{eq:circle-adjacent-identity},
\[
r\bigl(|J_{j-1}|^{2s}-|J_j|^{2s}\bigr)
\ge
\frac1{sj-1}
\left(\frac{r^2W^s}{2}\right)'.
\]
The boundary term vanishes at both endpoints: near the origin the defining
series gives $J_{j-1}(r)=O(r^{j-1})$ and $J_j(r)=O(r^j)$, while at infinity
$W(r)=O(r^{-1})$, so $r^2W(r)^s=O(r^{2-s})\to0$.  Integration therefore gives
$I_{j-1}(2s)\ge I_j(2s)$.  The strict part of
Lemma~\ref{lem:circle-scalar} applies on a nontrivial interval near the origin,
so the inequality is strict.  Taking $j=1,2,3$ proves
\eqref{eq:circle-low-chain}.
\end{proof}

\subsubsection{High modes near the critical exponent}
For $k\ge4$ the main difficulty is concentrated near $q=4$, where the tail of
every mode has the same critical divergence.  We therefore first prove the
comparison on $4<q\le5$ by following the integral from the origin to infinity.
The proof itself determines which auxiliary estimate is needed on each part of
the radial axis.

\begin{proposition}\label{prop:circle-high-critical}
For every integer $k\ge4$ and every $4<q\le5$,
\[
I_k(q)<I_0(q).
\]
\end{proposition}

\begin{proof}
Write $q=4+t$, $0<t\le1$, and set
\[
c:=\frac75,
\qquad
v_k:=ck^{2/3},
\qquad
R_k:=\sqrt{k^2+v_k^2}.
\]
We begin by decomposing the $k$th moment into three radial regions.  The estimates of these three pieces will later be combined with a fourth error term coming from the zero-mode tail.

First split
\[
I_k(4+t)
=
\int_0^k+\int_k^{R_k}+\int_{R_k}^\infty
|J_k(r)|^{4+t}r\,dr.
\]
The first piece is the pre-turning region.  To control it uniformly in the order, we need the following estimate exactly here.

\begin{lemma}\label{lem:pre-turning}
For every integer $k\ge4$ and $0<r\le k$,
\begin{equation}\label{eq:pre-turning}
|J_k(r)|
<
\frac{21}{40}k^{-1/3}
\exp\left(-\frac45(k-r)k^{-1/3}\right).
\end{equation}
Consequently, if $0\le t\le1$ and $d_0=21/40$, then
\begin{equation}\label{eq:pre-turning-integral}
\int_0^k|J_k(r)|^{4+t}r\,dr
<\frac3{125}\bigl(d_0k^{-1/3}\bigr)^t.
\end{equation}
\end{lemma}

\begin{proof}
We start from the integer-order integral representation
\eqref{eq:Bessel-integral}.  Put
\[
\eta:=k^{2/3},
\qquad
\gamma:=\operatorname{arsinh}(\eta/r).
\]
Since the integrand is entire in the angular variable and $k$ is an
integer, the two vertical sides cancel when the contour
$[-\pi,\pi]$ is shifted to $[-\pi+i\gamma,\pi+i\gamma]$.  Writing
$\theta=t+i\gamma$, we have $r\sinh\gamma=\eta$, and therefore
\[
\operatorname{Re}\,i(k\theta-r\sin\theta)
=
-k\gamma+\eta\cos t.
\]
We write $\mathcal I_\mu$ for the modified Bessel function of the first kind of order $\mu$, defined for $z>0$ by the convergent series
\[
\mathcal I_\mu(z)
:=
\sum_{m=0}^\infty
\frac{1}{m!\,\Gamma(m+\mu+1)}
\left(\frac z2\right)^{2m+\mu}.
\]
In particular,
\[
\mathcal I_0(z)
=
\frac1{2\pi}\int_{-\pi}^{\pi}e^{z\cos t}\,dt.
\]
Thus
\[
\begin{aligned}
|J_k(r)|
&\le
e^{-k\gamma}\frac1{2\pi}
\int_{-\pi}^{\pi}e^{\eta\cos t}\,dt\\
&=
e^{-k\gamma}\mathcal I_0(\eta)
=
e^{\eta-k\gamma}\eta^{-1/2}f(\eta),
\end{aligned}
\]
where
\[
f(z):=\sqrt z\,e^{-z}\mathcal I_0(z).
\]
The calligraphic notation $\mathcal I_\mu$ is used here to distinguish the modified Bessel functions from the modal moments $I_k(q)$ used throughout the circle argument.

We next control $f$.  Substitution of the modified Bessel equation for
$\mathcal I_0$ gives
\[
f''+2f'+\frac{f}{4z^2}=0.
\]
Moreover
\[
\frac{f'(1)}{f(1)}
=
-\frac12+\frac{\mathcal I_1(1)}{\mathcal I_0(1)}<0,
\]
because the positive series for $\mathcal I_0$ and $\mathcal I_1$ give
$\mathcal I_1(1)<\mathcal I_0(1)/2$ term by term after the first term.  If $f'$ had a
first zero $z_0>1$, then the differential equation would give
$f''(z_0)=-f(z_0)/(4z_0^2)<0$, which is incompatible with $f'$ reaching
zero from below for the first time.  Hence $f$ is strictly decreasing
on $[1,\infty)$.

Since $k\ge4$, we have $\eta=k^{2/3}>2$ and therefore $f(\eta)<f(2)$.
The positive series
\[
\mathcal I_0(2)=\sum_{m=0}^\infty\frac1{(m!)^2}
\]
gives $\mathcal I_0(2)<57/25$; for example, after summing through $m=4$ the
remaining tail is bounded by a geometric series with ratio at most
$1/36$.  Using also $\sqrt2<99/70$ and
$e^2>\sum_{m=0}^{7}2^m/m!=155/21$, we obtain
\[
f(\eta)<f(2)
=
\sqrt2\,e^{-2}\mathcal I_0(2)
<
\frac{11}{25}.
\]

It remains to estimate the exponential factor.  Define
\[
F_k(r):=\eta-k\operatorname{arsinh}(\eta/r).
\]
At $r=k$, with $u=k^{-1/3}$,
\[
F_k(k)=k\bigl(u-\operatorname{arsinh}u\bigr)
\le\frac16,
\]
because
\[
u-\operatorname{arsinh}u
=
\int_0^u\left(1-\frac1{\sqrt{1+s^2}}\right)ds
\le\int_0^u\frac{s^2}{2}\,ds
=\frac{u^3}{6}.
\]
Furthermore,
\[
F_k'(r)
=
\frac{k\eta}{r\sqrt{r^2+\eta^2}}
\ge
\frac{\eta}{\sqrt{k^2+\eta^2}}
=
\frac{k^{-1/3}}{\sqrt{1+k^{-2/3}}}
>
\frac45\,k^{-1/3},
\qquad 0<r\le k,
\]
where the last inequality follows from $k\ge4$.  Integrating from $r$
to $k$ gives
\[
F_k(r)
\le
\frac16-\frac45(k-r)k^{-1/3}.
\]
Finally, the positive exponential series gives
$e^{1/6}<105/88$, and hence
\[
\frac{11}{25}e^{1/6}<\frac{21}{40}.
\]
Combining the last estimates with the contour bound proves
\eqref{eq:pre-turning}.

For the integral consequence, raise \eqref{eq:pre-turning} to the power
$4+t$ and use $r\le k$:
\[
\begin{aligned}
\int_0^k|J_k(r)|^{4+t}r\,dr
&<
d_0^{\,4+t}k^{-(4+t)/3}
\int_0^k
e^{-\frac45(4+t)(k-r)k^{-1/3}}r\,dr\\
&\le
d_0^{\,4+t}k^{1-(4+t)/3}
\int_0^\infty
e^{-\frac45(4+t)s\,k^{-1/3}}\,ds\\
&=
\frac{5d_0^4}{4(4+t)}
\bigl(d_0k^{-1/3}\bigr)^t.
\end{aligned}
\]
Since $0\le t\le1$,
\[
\frac{5d_0^4}{4(4+t)}
\le
\frac5{16}\left(\frac{21}{40}\right)^4
<
\frac3{125},
\]
which proves \eqref{eq:pre-turning-integral}.
\end{proof}

Returning to the first piece, Lemma~\ref{lem:pre-turning} gives
\[
E_{1,k}(t):=
\int_0^k|J_k(r)|^{4+t}r\,dr
<
\frac3{125}\bigl(d_0k^{-1/3}\bigr)^t,
\qquad d_0=\frac{21}{40}.
\]
On the transition interval, Landau's bound \eqref{eq:Landau-amplitude}
and $R_k^2-k^2=v_k^2$ give
\[
\begin{aligned}
E_{2,k}(t)
&:=
\int_k^{R_k}|J_k(r)|^{4+t}r\,dr\\
&<
\left(\frac7{10}\right)^{4+t}
k^{-(4+t)/3}\frac{v_k^2}{2}\\
&=
\frac{c^2}{2}\left(\frac7{10}\right)^4
\left(\frac7{10}k^{-1/3}\right)^t.
\end{aligned}
\]
The third piece is the oscillatory tail.  This is where the common $t^{-1}$ singularity appears, so we must compare the exact phases before making any rough estimate.  The zero-order amplitude requires one quantitative refinement.

\begin{lemma}\label{lem:B0}
Recall that $B_0$ and $b_{0,q}$ are the $\nu=0$ specializations of the
quantities defined in \eqref{eq:Hankel-amplitude} and \eqref{eq:bnuq},
respectively.  For $r>0$ define
\[
h(r):=\frac{\pi r}{2}B_0(r)^2,
\qquad\text{equivalently}\qquad
B_0(r)^2=\frac{2h(r)}{\pi r}.
\]
Then
\begin{equation}\label{eq:h-bounds}
1-\frac1{8r^2}\le h(r)\le1,
\qquad
0<h'(r)\le\frac1{4r^3}.
\end{equation}
For $q\ge4$, $b_{0,q}$ is decreasing on $[1,\infty)$ and tends to zero as $r\to\infty$.
\end{lemma}

\begin{proof}
The two-sided bound for $h$ is classical.  In the notation of
\cite[\S10.18]{DLMF}, $B_0$ is the modulus function $M_0$, and the expansion
\cite[Eq.~10.18.17 and the signed-remainder statement in \S10.18(iii)]{DLMF},
specialized to order zero, gives
\[
1-\frac1{8r^2}\le \frac{\pi r}{2}B_0(r)^2\le1,
\qquad r>0;
\]
see also \cite[pp.~448--449]{Watson}.  Thus only the derivative estimate in
\eqref{eq:h-bounds} needs to be recorded.

Indeed, Nicholson's formula \eqref{eq:Nicholson} at $\nu=0$, followed by
$z=r\sinh u$, gives
\[
h(r)=\frac4\pi\int_0^\infty
\frac{K_0(2z)}{\sqrt{1+(z/r)^2}}\,dz.
\]
Differentiation under the integral sign yields
\[
h'(r)
=\frac4\pi\int_0^\infty
K_0(2z)\frac{z^2}{r^3}
\left(1+\frac{z^2}{r^2}\right)^{-3/2}dz>0.
\]
Moreover,
\[
\int_0^\infty z^2K_0(2z)\,dz
=\frac14\int_0^\infty\operatorname{sech}^3t\,dt
=\frac{\pi}{16},
\]
where we used the positive integral representation of $K_0$ and Tonelli's
theorem.  Hence
\[
0<h'(r)\le\frac4{\pi r^3}\int_0^\infty z^2K_0(2z)\,dz
=\frac1{4r^3}.
\]
This proves \eqref{eq:h-bounds}.

Since $B_0(r)^2=2h(r)/(\pi r)$,
\[
\frac{r(B_0^2)'}{B_0^2}
=-1+\frac{rh'}h.
\]
For $r\ge1$, \eqref{eq:h-bounds} gives
\[
\frac{rh'}h
\le
\frac{1/(4r^2)}{1-1/(8r^2)}
=\frac2{8r^2-1}.
\]
Taking the logarithmic derivative of
$b_{0,q}=(\pi/2)r^2B_0^{q+2}$ therefore yields
\[
\frac{rb_{0,q}'}{b_{0,q}}
\le
1-\frac q2+\frac{q+2}{8r^2-1}<0,
\qquad r\ge1,\quad q\ge4.
\]
Thus $b_{0,q}$ is decreasing on $[1,\infty)$.  Finally, $h(r)\le1$ gives
$B_0(r)^2\le2/(\pi r)$, and hence
$b_{0,q}(r)=O(r^{1-q/2})\to0$.  This proves the lemma.
\end{proof}

We now introduce the exact-phase tail calculation needed for this third piece.  The quantities $m(p)$ and $A(p)$ were fixed in Section~\ref{sec:preliminaries}, while $B_\nu$, $\theta_\nu$, and $b_{\nu,q}$ were introduced above in the higher-dimensional argument.

The main difficulty for $k\ge4$ is the common pole as $q\to4^+$.  We keep this pole explicit and compare the two tails before estimating them.

For $4\le q\le5$, let
\begin{equation}\label{eq:Gq}
G_q(u):=\int_0^u\bigl(|\cos w|^q-m(q)\bigr)\,dw.
\end{equation}
Here $m(q)$ is the cosine moment introduced in \eqref{eq:m-A}.  Since
$|\cos w|^q$ is $\pi$-periodic and
\[
\int_0^{\pi/2}\bigl(\cos^q w-m(q)\bigr)\,dw=0,
\]
the primitive $G_q$ is $\pi$-periodic and
$G_q(\pi/2+n\pi)=0$.  On $[0,\pi/2]$ the integrand is decreasing, hence
$G_q$ first increases and then decreases and is nonnegative.  Moreover,
for $4\le q\le5$,
\[
\cos^q w\le\cos^4 w,
\qquad
m(q)\ge m(5)>\frac13
\]
by \eqref{eq:cos-moments}.  Therefore
\[
0\le G_q(u)
\le
\Psi(u):=
\int_0^u\left(\cos^4w-\frac13\right)\,dw
=
\frac{u}{24}+\frac{\sin2u}{4}+\frac{\sin4u}{32}.
\]
The derivative $\Psi'(u)=\cos^4u-1/3$ has a unique zero
$u_*\in(0,\pi/4)$, so $\Psi$ attains its maximum there.  Since
$\cos^2u_*=1/\sqrt3$,
\[
\sin^2(4u_*)
=
16\frac1{\sqrt3}
\left(1-\frac1{\sqrt3}\right)
\left(\frac2{\sqrt3}-1\right)^2
=
\frac{176\sqrt3-304}{9}
<
\frac19,
\]
where the last inequality is equivalent to $176\sqrt3<305$.  Thus
$\sin4u_*<1/3$, and with $\pi<22/7$,
\[
\Psi(u_*)
<
\frac{\pi}{96}+\frac14+\frac1{96}
<
\frac3{10}.
\]
By the symmetry on the second half-period the same bound controls the
absolute value of $G_q$, and hence
\begin{equation}\label{eq:Gq-bound}
\|G_q\|_\infty<\frac3{10}.
\end{equation}

We now derive the exact phase integration-by-parts formula used in the
tail comparison.  Recall that $B_\nu$ and $\theta_\nu$ were defined in
\eqref{eq:Hankel-amplitude}, while $b_{\nu,q}$ was defined in
\eqref{eq:bnuq}.  Differentiating \eqref{eq:Gq} along the phase gives
\[
\frac{d}{dr}G_q(\theta_\nu(r))
=
\bigl(|\cos\theta_\nu(r)|^q-m(q)\bigr)\theta_\nu'(r).
\]
Using the exact phase derivative \eqref{eq:phase-derivative},
\[
\begin{aligned}
b_{\nu,q}(r)\frac{d}{dr}G_q(\theta_\nu(r))
&=
\frac{\pi}{2}r^2B_\nu(r)^{q+2}
\frac{2}{\pi rB_\nu(r)^2}
\bigl(|\cos\theta_\nu(r)|^q-m(q)\bigr)\\
&=
rB_\nu(r)^q
\bigl(|\cos\theta_\nu(r)|^q-m(q)\bigr).
\end{aligned}
\]
Integrating this identity over $[R,S]$ and integrating by parts yields
\[
\int_R^S rB_\nu(r)^q
\bigl(|\cos\theta_\nu(r)|^q-m(q)\bigr)\,dr
=
\bigl[b_{\nu,q}G_q(\theta_\nu)\bigr]_R^S
-
\int_R^S b_{\nu,q}'(r)G_q(\theta_\nu(r))\,dr.
\]
If $b_{\nu,q}$ is decreasing and tends to zero, then $G_q$ is bounded
and the upper endpoint tends to zero as $S\to\infty$.  We therefore obtain
\begin{equation}\label{eq:phase-integration-by-parts}
\int_R^\infty
rB_\nu(r)^q\bigl(|\cos\theta_\nu(r)|^q-m(q)\bigr)\,dr
=
-b_{\nu,q}(R)G_q(\theta_\nu(R))
-
\int_R^\infty b_{\nu,q}'(r)G_q(\theta_\nu(r))\,dr.
\end{equation}
The hypotheses on $b_{\nu,q}$ are supplied by
Lemma~\ref{lem:Hankel-bounds} when $\nu\ge1/2$ and $q>2$, and by
Lemma~\ref{lem:B0} for $\nu=0$ and $q\ge4$.  Thus the oscillatory
remainder is controlled by the endpoint value of $b_{\nu,q}$ without
replacing the exact phase by an approximate trigonometric phase.

The preceding identity leads to the following tail comparison, which we state and prove at the point where it is needed.

\begin{proposition}\label{prop:circle-tail}
Let $4<q\le5$, $t:=q-4$, and $k\ge1$.  If $v>0$ and
\[
R:=\sqrt{k^2+v^2},
\]
then
\begin{equation}\label{eq:circle-tail-upper}
\int_R^\infty |J_k(r)|^q r\,dr
\le
\frac{A(q)}{t}v^{-t/2}
+\frac35\left(\frac{2}{\pi}\right)^{q/2}\bigl(v^{1-q/2}+k^2v^{-1-q/2}\bigr).
\end{equation}
At the first positive zero $\zeta_0$ of $J_0$,
\begin{equation}\label{eq:circle-tail-lower}
\int_{\zeta_0}^\infty |J_0(r)|^q r\,dr
\ge
\frac{A(q)}{t}\zeta_0^{-t/2}
-\left(\frac{2}{\pi}\right)^{q/2}\left(
\frac3{10}\zeta_0^{1-q/2}
+\frac{m(q)}8\zeta_0^{-q/2}
\right).
\end{equation}
\end{proposition}

The quantities $m(q)$ and $A(q)$ in \eqref{eq:circle-tail-upper}--\eqref{eq:circle-tail-lower} are the functions introduced in \eqref{eq:m-A}.

\begin{proof}
We separate the mean contribution from the oscillatory remainder.  Since
\[
J_k(r)=B_k(r)\cos\theta_k(r),
\]
we have
\[
\int_R^\infty |J_k(r)|^q r\,dr
=
m(q)\int_R^\infty rB_k(r)^q\,dr
+
\int_R^\infty rB_k(r)^q
\bigl(|\cos\theta_k(r)|^q-m(q)\bigr)\,dr.
\]
The second integral is exactly the quantity in
\eqref{eq:phase-integration-by-parts}.  Since $k\ge1$, Lemma~\ref{lem:Hankel-bounds}
shows that $b_{k,q}$ is decreasing and tends to zero.  Therefore
\eqref{eq:Gq-bound} gives
\[
\left|
\int_R^\infty rB_k(r)^q
\bigl(|\cos\theta_k(r)|^q-m(q)\bigr)\,dr
\right|
\le
\frac35\,b_{k,q}(R).
\]
Here we used
\[
|b_{k,q}(R)G_q(\theta_k(R))|
\le\frac3{10}b_{k,q}(R)
\]
and
\[
\left|\int_R^\infty b_{k,q}'(r)G_q(\theta_k(r))\,dr\right|
\le
\frac3{10}\int_R^\infty(-b_{k,q}'(r))\,dr
=
\frac3{10}b_{k,q}(R).
\]

For $k\ge1$, Lemma \ref{lem:Hankel-bounds} yields
\[
B_k(r)^2\le\frac{2}{\pi\sqrt{r^2-k^2}},
\qquad r>k.
\]
Hence, with $R=\sqrt{k^2+v^2}$,
\[
\begin{aligned}
m(q)\int_R^\infty rB_k(r)^q\,dr
&\le
m(q)\left(\frac2\pi\right)^{q/2}
\int_R^\infty r(r^2-k^2)^{-q/4}\,dr.
\end{aligned}
\]
Set $z:=\sqrt{r^2-k^2}$, so that $r\,dr=z\,dz$ and the lower endpoint becomes $v$.  Since $q=4+t$,
\[
\int_v^\infty z^{1-q/2}\,dz
=\frac{2}{t}v^{-t/2}.
\]
Thus the mean term is exactly bounded by
\[
\frac{A(q)}{t}v^{-t/2},
\qquad
A(q)=2\left(\frac2\pi\right)^{q/2}m(q).
\]
For the oscillatory error, the same amplitude bound at $R$ gives
\[
\begin{aligned}
b_{k,q}(R)
&=\frac\pi2R^2B_k(R)^{q+2}\\
&\le
\left(\frac2\pi\right)^{q/2}
R^2v^{-1-q/2}\\
&=
\left(\frac2\pi\right)^{q/2}
\bigl(v^{1-q/2}+k^2v^{-1-q/2}\bigr).
\end{aligned}
\]
Combining the last three displays proves \eqref{eq:circle-tail-upper}.

For the zero mode, Lemma~\ref{lem:B0} guarantees that $b_{0,q}$ is decreasing
and tends to zero for $q\ge4$, so \eqref{eq:phase-integration-by-parts} is
applicable.  Start it at the first positive zero $\zeta_0$ of $J_0$.  Because
\[
J_0(\zeta_0)=B_0(\zeta_0)\cos\theta_0(\zeta_0)=0
\]
and $B_0(\zeta_0)>0$, we have $\theta_0(\zeta_0)=\pi/2$ modulo $\pi$.  Hence $G_q(\theta_0(\zeta_0))=0$, so the endpoint term in \eqref{eq:phase-integration-by-parts} vanishes exactly.  The remaining oscillatory term is bounded below by
\[
-\frac3{10}b_{0,q}(\zeta_0).
\]
By Lemma \ref{lem:B0},
\[
B_0(r)^2=\frac{2h(r)}{\pi r},
\qquad
h(r)\ge1-\frac1{8r^2}.
\]
Since $q/2\ge2$, Bernoulli's inequality $(1-x)^{q/2}\ge1-(q/2)x$ for the present range gives
\[
B_0(r)^q
\ge
\left(\frac{2}{\pi}\right)^{q/2}r^{-q/2}
\left(1-\frac{q}{16r^2}\right),
\qquad r\ge\zeta_0.
\]
Integrating the mean part from $\zeta_0$ to infinity produces
\[
\frac{A(q)}{t}\zeta_0^{-t/2}
-
\left(\frac2\pi\right)^{q/2}
\frac{m(q)}8\zeta_0^{-q/2},
\]
while $b_{0,q}(\zeta_0)\le(2/\pi)^{q/2}\zeta_0^{1-q/2}$ follows from $h\le1$.  These two estimates give \eqref{eq:circle-tail-lower}.
\end{proof}

The two tail estimates have the same singular factor.  Before returning to the
main comparison we record the cancellation in the exact form in which it will
be used:
\begin{equation}\label{eq:common-pole-cancel}
\frac{A(q)}{t}\bigl(\zeta_0^{-t/2}-v^{-t/2}\bigr)
=\left(\frac{2}{\pi}\right)^{q/2}m(q)
\int_{\zeta_0}^v u^{1-q/2}\,du.
\end{equation}
Thus no divergent fourth moment is interpolated; the common pole is removed
before the remaining estimates are made.

Returning to the $k$th mode, recall from the proof of Proposition~\ref{prop:circle-high-critical} that $v_k=\frac75 k^{2/3}$ and $R_k=\sqrt{k^2+v_k^2}$.  Applying Proposition~\ref{prop:circle-tail} with $v=v_k$ yields
\[
\int_{R_k}^\infty|J_k(r)|^{4+t}r\,dr
\le
\frac{A(4+t)}{t}v_k^{-t/2}
+E_{3,k}(t),
\]
where
\[
E_{3,k}(t)
=
\frac35\left(\frac2\pi\right)^{2+t/2}
\left(
v_k^{-1-t/2}
+k^2v_k^{-3-t/2}
\right).
\]

To compare this with the zero mode, we next need only a fixed positive core and a crude localization of the first zero.  We record that elementary information now.

\begin{lemma}\label{lem:J0-core}
For $0\le x\le1$,
\begin{equation}\label{eq:J0-core}
0\le
P(x):=1-x+\frac{x^2}{4}-\frac{x^3}{36}
\le J_0(2\sqrt x)
\le1-x+\frac{x^2}{4}
\le e^{-x}.
\end{equation}
Let $\zeta_0$ denote the first positive zero of $J_0$.  Then
\begin{equation}\label{eq:zeta0-bounds}
\frac{12}{5}<\zeta_0<\frac52.
\end{equation}
Moreover,
\begin{equation}\label{eq:I0-lower}
I_0(q)\ge\frac2{q+1},
\qquad q>0,
\end{equation}
whenever the left-hand side is interpreted as an extended integral.
\end{lemma}

\begin{proof}
The alternating series
\[
J_0(2\sqrt x)=\sum_{m=0}^\infty\frac{(-1)^m x^m}{(m!)^2}
\]
gives the two polynomial bounds in \eqref{eq:J0-core}: after the displayed
terms the absolute values of the remaining summands decrease on
$0\le x\le1$.  The polynomial $P$ is nonnegative on $[0,1]$, so the
leftmost inequality follows as well.

The first positive zero is the classical tabulated value
\[
\zeta_0=j_{0,1}=2.404825557\ldots;
\]
see \cite[Chapter~XV]{Watson} and \cite[\S10.75(iii)]{DLMF}.  In
particular, \eqref{eq:zeta0-bounds} follows immediately.

Finally, the integral representation \eqref{eq:Bessel-integral} at
$k=0$, after taking the real part and using symmetry, gives
\[
J_0(r)=\frac1\pi\int_0^\pi\cos(r\sin\theta)\,d\theta.
\]
Since $\cos u\ge1-u^2/2$ for every real $u$,
\[
\begin{aligned}
J_0(r)
&\ge
\frac1\pi\int_0^\pi
\left(1-\frac{r^2\sin^2\theta}{2}\right)\,d\theta\\
&=
1-\frac{r^2}{4},
\qquad 0\le r\le2.
\end{aligned}
\]
The right-hand side is nonnegative on $[0,2]$.  Therefore, with the
change of variables $u=1-r^2/4$,
\[
\begin{aligned}
I_0(q)
&\ge
\int_0^2\left(1-\frac{r^2}{4}\right)^q r\,dr\\
&=
2\int_0^1u^q\,du
=
\frac2{q+1},
\end{aligned}
\]
which proves \eqref{eq:I0-lower}.
\end{proof}

For the zero mode, let
\[
Z_0(p;2):=\int_0^2|J_0(r)|^p r\,dr.
\]
Since $\zeta_0>2$, Proposition~\ref{prop:circle-tail} gives
\[
I_0(4+t)
\ge
Z_0(4+t;2)
+\frac{A(4+t)}{t}\zeta_0^{-t/2}
-E_4(t),
\]
where
\[
E_4(t)
=
\left(\frac2\pi\right)^{2+t/2}
\left(
\frac3{10}\zeta_0^{-1-t/2}
+\frac{m(4+t)}8\zeta_0^{-2-t/2}
\right).
\]
Subtracting the upper bound for $I_k$ from this lower bound for $I_0$
and using \eqref{eq:common-pole-cancel}, we obtain
\[
\begin{aligned}
I_0(4+t)-I_k(4+t)
\ge\;&
Z_0(4+t;2)\\
&+
\left(\frac2\pi\right)^{2+t/2}m(4+t)
\int_{\zeta_0}^{v_k}u^{-1-t/2}\,du\\
&-
\bigl(E_{1,k}(t)+E_{2,k}(t)+E_{3,k}(t)+E_4(t)\bigr).
\end{aligned}
\]
By Lemma~\ref{lem:J0-core}, $\zeta_0<5/2$, while
\[
v_k\ge\frac75\,4^{2/3}>\frac72
\qquad(k\ge4).
\]
Hence $[5/2,7/2]\subset[\zeta_0,v_k]$, and it is enough to retain this
fixed part of the positive integral.  Define
\begin{equation}\label{eq:F-circle}
F(p)
:=
Z_0(p;2)
+\left(\frac2\pi\right)^{p/2}m(p)
\int_{5/2}^{7/2}u^{1-p/2}\,du
\end{equation}
and set $H_k(t):=E_{1,k}(t)+E_{2,k}(t)+E_{3,k}(t)+E_4(t)$. Then
\begin{equation}\label{eq:circle-FH}
I_0(4+t)-I_k(4+t)
\ge
F(4+t)-H_k(t).
\end{equation}

We next record the uniform estimate for the error.  Each $t$-dependent
base occurring above is at most $8/15$.  Indeed, for $k\ge4$,
\[
k^{-1/3}<\frac23,
\qquad
k^{-2/3}<\frac25,
\qquad
v_k>\frac72,
\]
and, using $\pi>25/8$ and $\zeta_0>12/5$,
\[
d_0k^{-1/3}<\frac8{15},
\quad
\frac7{10}k^{-1/3}<\frac8{15},
\quad
\sqrt{\frac{2}{\pi v_k}}<\frac8{15},
\quad
\sqrt{\frac{2}{\pi\zeta_0}}<\frac8{15}.
\]
Also $m(4+t)\le m(4)=3/8$.  To make the numerical estimate transparent, first factor the $t$-dependent bases in the last two errors:
\[
\begin{aligned}
E_{3,k}(t)
&=\frac35\left(\frac2\pi\right)^2
\bigl(v_k^{-1}+k^2v_k^{-3}\bigr)
\left(\sqrt{\frac{2}{\pi v_k}}\right)^t,\\
E_4(t)
&\le\left(\frac2\pi\right)^2
\left(\frac3{10}\zeta_0^{-1}+\frac3{64}\zeta_0^{-2}\right)
\left(\sqrt{\frac{2}{\pi\zeta_0}}\right)^t.
\end{aligned}
\]
Using $v_k^{-1}=c^{-1}k^{-2/3}$, $k^2v_k^{-3}=c^{-3}$, together with the bounds above, gives
\[
\begin{aligned}
E_{1,k}(t)&<\frac3{125}\left(\frac8{15}\right)^t,
&
E_{2,k}(t)&<\frac{13}{55}\left(\frac8{15}\right)^t,\\
E_{3,k}(t)&<\frac4{25}\left(\frac8{15}\right)^t,
&
E_4(t)&<\frac3{55}\left(\frac8{15}\right)^t.
\end{aligned}
\]
Indeed, the last three numerical coefficients follow respectively from
\[
\frac{c^2}{2}\left(\frac7{10}\right)^4<\frac{13}{55},\qquad
\frac35\left(\frac2\pi\right)^2
\left(c^{-1}k^{-2/3}+c^{-3}\right)<\frac4{25},
\]
and
\[
\left(\frac2\pi\right)^2
\left(\frac3{10}\zeta_0^{-1}+\frac3{64}\zeta_0^{-2}\right)<\frac3{55}.
\]
Adding the four estimates yields
\begin{equation}\label{eq:Hk-bound}
H_k(t)<\frac{19}{40}\left(\frac8{15}\right)^t.
\end{equation}

The function $F$ is a genuine positive moment on a finite measure
space.  The first term in \eqref{eq:F-circle} is already the $p$th
moment of $|J_0|$ on $[0,2]$ with measure $r\,dr$.  For the second term,
using the definition of $m(p)$ gives the exact identity
\[
\begin{aligned}
&\left(\frac2\pi\right)^{p/2}m(p)
\int_{5/2}^{7/2}u^{1-p/2}\,du\\
&\qquad=
\int_{5/2}^{7/2}\int_0^\pi
\left(\sqrt{\frac2\pi}\,u^{-1/2}|\cos w|\right)^p
\frac{u\,dw\,du}{\pi}.
\end{aligned}
\]
To combine these two positive terms into a single moment, we place both
pieces in the plane and set
\[
\Omega_1:=[0,2]\times\{0\},
\qquad
\Omega_2:=\left[\frac52,\frac72\right]\times[0,\pi].
\]
Their first-coordinate ranges are disjoint, so
$\Omega_1\cap\Omega_2=\varnothing$.  In particular, the function below is
defined separately on the two components and no compatibility between the
two formulas is required: a point $(r,0)\in\Omega_1$ has $r\le2$ and hence
does not belong to $\Omega_2$, whereas a point $(u,0)\in\Omega_2$ has
$u\ge5/2$ and hence does not belong to $\Omega_1$.

Put $\Omega:=\Omega_1\cup\Omega_2$ and define a measure $\mu$ on $\Omega$ by
\[
\int_\Omega f\,d\mu
:=
\int_0^2 f(r,0)\,r\,dr
+
\frac1\pi\int_{5/2}^{7/2}\int_0^\pi f(u,w)\,u\,dw\,du.
\]
Define
\[
X(r,0):=|J_0(r)|,\qquad (r,0)\in\Omega_1,
\]
and
\[
X(u,w):=\sqrt{\frac2\pi}\,u^{-1/2}|\cos w|,
\qquad (u,w)\in\Omega_2.
\]
Then
\[
\begin{aligned}
\int_\Omega X^p\,d\mu
&=\int_0^2|J_0(r)|^p r\,dr\\
&\quad+
\frac1\pi\int_{5/2}^{7/2}\int_0^\pi
\left(\sqrt{\frac2\pi}\,u^{-1/2}|\cos w|\right)^p
u\,dw\,du\\
&=F(p).
\end{aligned}
\]
Moreover,
\[
\mu(\Omega)=\int_0^2r\,dr+
\frac1\pi\int_{5/2}^{7/2}\int_0^\pi u\,dw\,du
=2+3=5.
\]
Thus $F$ is indeed a positive moment on a finite measure space.  At this point the only remaining issue is a lower bound for the positive quantity $F(4+t)$ from its values at two fixed exponents.  This is exactly the standard log-convexity/extrapolation principle for moments, so we introduce it only now.

\begin{lemma}[Log-convexity of moments]\label{lem:moments}
Let $Y\ge0$ be measurable on a measure space $(\mathcal X,\mu)$ and put
$M(p):=\int_{\mathcal X}Y^p\,d\mu$.  Whenever the moments appearing below are finite,
\begin{equation}\label{eq:logconvex-moments}
M((1-\theta)p_0+\theta p_1)
\le
M(p_0)^{1-\theta}M(p_1)^\theta,
\qquad 0\le\theta\le1.
\end{equation}
If $M(3),M(4)>0$, then
\begin{equation}\label{eq:moment-extrapolation}
M(4+t)
\ge
\frac{M(4)^{1+t}}{M(3)^t},
\qquad t\ge0.
\end{equation}
\end{lemma}

The log-convexity in \eqref{eq:logconvex-moments} is the classical H\"older interpolation inequality for moments; see, for example, \cite[Chapter~2]{HardyLittlewoodPolya} or \cite[Chapter~6]{MitrinovicPecaricFink}.  The extrapolation formula is the same statement with
\[
4=\frac{t}{1+t}\,3+\frac1{1+t}(4+t),
\]
followed by a rearrangement.  This proves Lemma~\ref{lem:moments}; we record it
only in the form needed later.

All components of the high-mode decomposition are now available: the
pre-turning estimate controls $[0,k]$, Landau's bound controls the transition
interval, Proposition~\ref{prop:circle-tail} compares the exact-phase tails,
and Lemma~\ref{lem:J0-core} supplies a fixed positive zero-mode core.  We now
assemble them.

Therefore Lemma~\ref{lem:moments}, applied with $\mathcal X=\Omega$ and $Y=X$, gives
\[
F(4+t)
\ge
\frac{F(4)^{1+t}}{F(3)^t}.
\]
The core estimate \eqref{eq:J0-core}, together with
\eqref{eq:cos-moments}, yields
\[
F(4)>\frac{12}{25},
\qquad
F(3)<\frac56.
\]
Hence
\begin{equation}\label{eq:F-growth}
F(4+t)
>
\frac{12}{25}
\left(\frac{72}{125}\right)^t.
\end{equation}
Since $12/25>19/40$ and $72/125>8/15$,
the lower bound \eqref{eq:F-growth} for the positive contribution is
strictly larger than the error estimate coming from
\eqref{eq:circle-FH} and \eqref{eq:Hk-bound}.  More precisely,
\[
I_0(4+t)-I_k(4+t)
\ge F(4+t)-H_k(t)>0,
\]
for every $k\ge4$ and $0<t\le1$.  This proves Proposition~\ref{prop:circle-high-critical}.
\end{proof}

\subsubsection{Completion of the interval $5\le q<6$}

Only the high modes remain in this interval.  At $q=5$, the quantitative
estimates in the proof of Proposition~\ref{prop:circle-high-critical} give,
for every $k\ge4$,
\begin{equation}\label{eq:Ik5}
I_k(5)<\frac{13}{40}.
\end{equation}
By the uniform Landau bound \eqref{eq:Landau-amplitude},
\[
I_k(q)
<
\frac{13}{40}\left(\frac7{15}\right)^{q-5},
\qquad 5\le q<6,
\]
while Lemma~\ref{lem:J0-core} gives $I_0(q)\ge2/(q+1)$.  Hence
\begin{equation}\label{eq:circle-high-large-q}
\frac{I_k(q)}{I_0(q)}
<
\frac{13(q+1)}{80}
\left(\frac7{15}\right)^{q-5}<1,
\qquad k\ge4,\quad 5\le q<6.
\end{equation}

We can now finish the circle case.  Proposition~\ref{prop:circle-low-modes}
covers $j=1,2,3$ for $4<q<6$; Proposition~\ref{prop:circle-high-critical}
covers $k\ge4$ for $4<q\le5$; and
\eqref{eq:circle-high-large-q} covers $k\ge4$ for $5\le q<6$.
Together with the previously known range \eqref{eq:circle-known-large-q},
this proves \eqref{eq:circle-goal} for every $q>4$ and completes the proof of
Theorem~\ref{thm:main-extension} in dimension two.

\section{Proof of Theorem~\ref{thm:intro-global-stability}}\label{sec:global-stability}

Theorem~\ref{thm:main-extension} has already settled the modal comparison, so
no new Bessel estimate is needed here.  We therefore start the proof from the
deficit itself and ask how it splits when the $L^2$ mass is distributed among
the harmonic degrees.  The exact decomposition will expose the spectral gap
immediately; after the lower bound is obtained, the only remaining issue will
be the equality case.

Fix a finite admissible $q$ and put
\[
p:=\frac q2,
\qquad K_d:=(2\pi)^d,
\qquad u_k(r):=r^{-\nu}J_{\nu+k}(r),
\qquad \phi_k(r):=u_k(r)^2,
\]
\[
a_k:=\|P_kf\|_2^2,
\qquad I_k:=I_{\nu,k}(q),
\qquad
\lambda_k:=\|\phi_k\|_{L^p(r^{d-1}dr)}=I_k^{2/q}.
\]
Stempak's weighted Bessel estimate \cite{Stempak} (see also
\cite[Theorem~1(i)]{COS}) gives $\lambda_k\to0$.  Hence
\begin{equation}\label{eq:lambda-star}
\lambda_*:=\max_{k\ge1}\lambda_k<\lambda_0,
\qquad
\mathcal K_*:=\{k\ge1:\lambda_k=\lambda_*\}
\end{equation}
is a finite nonempty set.  When $d\ge3$, Theorem~\ref{thm:main-extension}
gives $\mathcal K_*=\{1\}$.

Let
\[
\mathcal M:=\{c\mathbf1:c\in\mathbb C\},
\qquad
\operatorname{dist}(f,\mathcal M)^2
=\sum_{k\ge1}a_k,
\]
and recall
\begin{equation}\label{eq:deficit}
D_{d,q}(f)
:=C_{d,q}^2\|f\|_2^2
-\|\widehat{f\sigma}\|_{L^q_{\mathrm{rad}}L^2_{\mathrm{ang}}}^2,
\qquad
C_{d,q}^2=K_d\lambda_0.
\end{equation}
To obtain the exact decomposition, the only additional quantity we need is the
remainder in Minkowski's inequality.  We introduce it precisely at the point
where it enters the deficit:
\begin{equation}\label{eq:triangle-remainder}
\mathcal T_q(f)
:=
\sum_{k=0}^\infty a_k\lambda_k
-
\left\|\sum_{k=0}^\infty a_k\phi_k\right\|_{L^p(r^{d-1}dr)}
\ge0.
\end{equation}

\begin{proposition}[Exact deficit decomposition]\label{prop:exact-deficit}
For every $f\in L^2(\mathbb S^{d-1})$,
\begin{equation}\label{eq:exact-deficit-decomp}
\frac{D_{d,q}(f)}{K_d}
=
\sum_{k\ge1}(\lambda_0-\lambda_k)a_k
+\mathcal T_q(f).
\end{equation}
\end{proposition}

\begin{proof}
Insert the spectral identity \eqref{eq:spectral-identity} into
\eqref{eq:deficit}, add and subtract $K_d\sum_k a_k\lambda_k$, and use
$\sum_k a_k=\|f\|_2^2$.
\end{proof}

\begin{proof}[Proof of Theorem~\ref{thm:intro-global-stability}]
We now return to the theorem.  The decomposition
\eqref{eq:exact-deficit-decomp} already contains the desired lower bound.
Indeed, $\lambda_k\le\lambda_*$ for $k\ge1$ and
$\mathcal T_q(f)\ge0$, so
\[
D_{d,q}(f)
\ge K_d(\lambda_0-\lambda_*)\sum_{k\ge1}a_k
=\Gamma_{d,q}\|f-P_0f\|_2^2.
\]
A pure harmonic of any degree in $\mathcal K_*$ attains equality, so the
coefficient is optimal.

The inequality and the optimal coefficient are therefore settled.  It remains
only to identify equality.  This is the point at which strict convexity of
$L^p$ enters: since $p>1$, $L^p$ is strictly convex,
and equality in \eqref{eq:triangle-remainder} can hold only when all nonzero
profiles are proportional.  The functions $\phi_i$ and $\phi_j$ with
$i\ne j$ cannot be proportional because their expansions at the origin have
different first vanishing orders.  Hence equality forces a single harmonic
degree, which must be either zero or an element of $\mathcal K_*$.  In
particular, for $d\ge3$ the strict ordering from
Theorem~\ref{thm:main-extension} gives $\mathcal K_*=\{1\}$, and the only
nonconstant equality cases are first spherical harmonics.
\end{proof}

\begin{remark}\label{rem:circle-index}
When $d=2$, the optimal coefficient has the exact spectral form
\[
\Gamma_{2,q}
=(2\pi)^2
\left[
I_0(q)^{2/q}
-
\max_{k\ge1}I_k(q)^{2/q}
\right].
\]
The proof of Theorem~\ref{thm:main-extension} identifies degree zero as the
unique global maximizing degree, but it does not determine the largest
nonzero degree for every real $q>4$.  Thus the displayed formula is sharp even
though the maximizing nonzero index is not known in full generality.
\end{remark}

\section{Proof of Theorem~\ref{thm:intro-endpoint-stability}}

At the endpoint there is no radial integral to balance, so the finite-$q$
deficit decomposition is no longer the right starting point.  The spectral
identity shows that the proof has a single modal input: the suprema of
$r^{-\nu}J_{\nu+k}$ must decrease strictly with $k$.  We isolate and prove
that statement first; once it is available, the stability estimate and its
sharp coefficient follow in one line from the spectral identity.  Only after
the theorem is proved do we record the explicit description of the first
nonzero supremum.

Recall the notation $K_d:=(2\pi)^d$ from Section~\ref{sec:global-stability},
and set
\[
h_\nu:=\frac{2^{-\nu}}{\Gamma(\nu+1)}.
\]
We then define
\begin{equation}\label{eq:mu-k}
\mu_k
:=
\sup_{r\ge0}|r^{-\nu}J_{\nu+k}(r)|^2.
\end{equation}
At $k=0$ the value at the origin is understood by continuity, so
\[
\mu_0=h_\nu^2,
\qquad
C_{d,\infty}^2=K_d\mu_0=\omega_{d-1}^2.
\]

For the proof of Theorem~\ref{thm:intro-endpoint-stability}, it is therefore
enough to show
\[
\mu_0>\mu_1>\mu_2>\cdots.
\]
Thus the entire endpoint proof is reduced to the following lemma; Landau's
stationary-point principle is introduced here because this is precisely where
it is needed.

\begin{lemma}\label{lem:sup-ordering}
Fix $\nu\ge0$.  The function
\[
m(\alpha):=\sup_{r>0}|r^{-\nu}J_\alpha(r)|
\]
is strictly decreasing for $\alpha>\nu$ and extends continuously to $\alpha=\nu$ with value $h_\nu$.  Consequently
\begin{equation}\label{eq:mu-ordering}
\mu_0>\mu_1>\mu_2>\cdots.
\end{equation}
\end{lemma}

\begin{proof}
For $\alpha>\nu$, apply Landau's stationary-point monotonicity principle
\cite[Eq.~(8), pp.~150--151]{LandauEJDE} to
\[
F(r):=r^{-\nu},\qquad C_\alpha(r):=J_\alpha(r).
\]
Here $\mathcal A_\alpha$ is the positive kernel defined in
\eqref{eq:Malpha-Aalpha}.  Indeed,
\[
\frac{d}{dr}\!\left(F(r)^2\mathcal A_\alpha(r)\right)
=
\frac{d}{dr}\!\left(r^{-2\nu}\mathcal A_\alpha(r)\right)
=
-2\nu r^{-2\nu-1}\mathcal A_\alpha(r)
+r^{-2\nu}\mathcal A_\alpha'(r)<0,
\]
because $\mathcal A_\alpha(r)>0$ and $\mathcal A_\alpha'(r)<0$ on
$(0,\infty)$.  Thus the hypothesis of Landau's stationary-point principle is
satisfied.  It follows that the absolute value of
$r^{-\nu}J_\alpha(r)$ at every positive stationary point is strictly
decreasing with the order.

We also need to explain why this controls the supremum.  For
$\alpha>\nu$, the small-radius expansion
\[
J_\alpha(r)\sim \frac{(r/2)^\alpha}{\Gamma(\alpha+1)}
\qquad (r\to0^+)
\]
gives
\[
r^{-\nu}J_\alpha(r)\to0,
\]
while the large-radius asymptotic formula for Bessel functions gives
\[
r^{-\nu}J_\alpha(r)=O(r^{-\nu-1/2})\to0
\qquad (r\to\infty).
\]
Hence the supremum is attained at an interior stationary point.  Landau's
principle therefore applies directly and yields the strict decrease of
$m(\alpha)$.

At the endpoint, Poisson's bound \eqref{eq:Poisson-bound} gives
$m(\nu)\le h_\nu$, while the defining series gives
$r^{-\nu}J_\nu(r)\to h_\nu$ as $r\to0^+$.  Thus $m(\nu)=h_\nu$.
Continuity as $\alpha\to\nu$ follows from the standard compact-order
Bessel bounds; see \cite[Chapters~II and VII]{Watson} and
\cite[\S\S10.7 and 10.17]{DLMF}.
\end{proof}

The modal part is now finished, so we return immediately to the endpoint
deficit.  No further auxiliary estimate is needed for the sharp inequality.

\begin{proof}[Proof of Theorem~\ref{thm:intro-endpoint-stability}]
By \eqref{eq:spectral-identity} with the outer norm replaced by the supremum
and by Lemma~\ref{lem:sup-ordering}, which gives $\mu_k\le\mu_1$ for every
$k\ge1$,
\[
\|\widehat{f\sigma}\|_{L^\infty_{\mathrm{rad}}L^2_{\mathrm{ang}}}^2
=K_d\sup_{r\ge0}\sum_{k\ge0}a_k\phi_k(r)
\le
K_d\left(\mu_0a_0+\mu_1\sum_{k\ge1}a_k\right).
\]
This proves \eqref{eq:intro-endpoint-stability}; a pure first harmonic proves
optimality.  The strict inequalities $\mu_1>\mu_k$ for $k\ge2$ from
Lemma~\ref{lem:sup-ordering} exclude higher degrees in an equality case.  A mixture of degrees zero and one is also impossible: $\phi_0$ attains its maximum only at $r=0$, while $\phi_1(0)=0$ and $\mu_1>0$.
\end{proof}

Theorem~\ref{thm:intro-endpoint-stability} is now proved.  For completeness,
we finally record a useful description of the sharp nonconstant modal value.
The number $\mu_1$ can be expressed at a single critical point.  Let $\xi_d$ be the first positive critical point of
\[
u_1(r):=r^{-\nu}J_{\nu+1}(r).
\]
Then
\begin{equation}\label{eq:mu1-xi}
\xi_dJ_\nu(\xi_d)
=(d-1)J_{\nu+1}(\xi_d),
\qquad
\mu_1
=\xi_d^{-2\nu}J_{\nu+1}(\xi_d)^2.
\end{equation}
Indeed, $u_1$ solves
\[
u_1''+\frac{d-1}{r}u_1'+g(r)u_1=0,
\qquad
g(r):=1-\frac{d-1}{r^2}.
\]
The quantity
\[
u_1^2+\frac{(u_1')^2}{g}
\]
is nonincreasing for $r>\sqrt{d-1}$, so no later critical point can have a larger absolute value than the first one.

\section{Technical details}\label{app:limits}

This final section plays the role of an appendix and collects the routine analytic justifications, parameter checks, and normalization remarks used in the preceding proofs.

\subsection{Differentiation in the Bessel order}

Fix $q>1$, $\beta\in\mathbb R$, and $\alpha_0>0$ with
\[
\beta+q\alpha_0>-1,
\qquad
\beta-\frac q2<-1.
\]
Choose a compact positive-order interval $[\alpha_-,\alpha_+]$ containing $\alpha_0$ and satisfying $\beta+q\alpha_->-1$.  The standard compact-order Bessel estimates (see \cite[Chapters~II and VII]{Watson} and \cite[\S\S10.7 and 10.17]{DLMF}) give, uniformly in this interval,
\[
|J_\alpha(r)|\le Cr^{\alpha_-},
\qquad
|J_\alpha'(r)|\le Cr^{\alpha_--1}
\quad(0<r\le1),
\]
and
\[
|J_\alpha(r)|+|J_\alpha'(r)|\le Cr^{-1/2}
\quad(r\ge1).
\]
The positive representation \eqref{eq:Malpha-Aalpha} of $\mathcal A_\alpha$ gives
\[
\mathcal A_\alpha(r)\le C(1+|\log r|),
\qquad
-r\mathcal A_\alpha'(r)\le C
\quad(0<r\le1),
\]
and
\[
\mathcal A_\alpha(r)\le \frac Cr,
\qquad
-r\mathcal A_\alpha'(r)\le \frac Cr
\quad(r\ge1).
\]
Using \eqref{eq:order-variation}, it follows that
\[
\left|\partial_\alpha\bigl(r^\beta|J_\alpha(r)|^q\bigr)\right|
\le
\begin{cases}
Cr^{\beta+q\alpha_-}(1+|\log r|),&0<r\le1,\\
Cr^{\beta-q/2},&r\ge1.
\end{cases}
\]
Both functions are integrable.  This justifies differentiation under the integral in Lemma~\ref{lem:order-derivative} and also shows that the order derivative depends continuously on $\alpha$ on compact subintervals of the admissible region.

For the integration by parts used there, first integrate over $[\varepsilon,R]$.  The boundary term at zero is $O(r^{\beta+1+q\alpha_-}(1+|\log r|))$, while the boundary term at infinity is $O(r^{\beta-q/2})$.  Both vanish under the stated hypotheses, and the bulk terms are absolutely integrable by the same estimates.

\subsection{Positive kernels, phases, and zeros}

All radius derivatives of the positive kernels used above may be taken under the integral on compact positive-radius intervals.  Near the origin the logarithmic singularity of $K_0$ is harmless, while $xK_1(x)\le1$ controls the first derivative.  At infinity the modified Bessel functions decay exponentially.  The same observations justify the Tonelli interchanges made before a sign-changing difference is introduced.

The Wronskian guarantees
\[
J_\alpha(r)^2+Y_\alpha(r)^2>0,
\qquad r>0,
\]
so the Hankel phase $\theta_\alpha$ from \eqref{eq:Hankel-amplitude} may
be chosen continuously on $(0,\infty)$.  In the circle tail argument the
primitive $G_q$ is $\pi$-periodic and bounded; since $b_{\alpha,q}$ is defined
in \eqref{eq:bnuq} and $b_{\alpha,q}(r)\to0$,
\[
\int_R^\infty |b_{\alpha,q}'(r)|\,dr=b_{\alpha,q}(R),
\]
which justifies the integration by parts at infinity.

No logarithmic derivative of $J_\alpha$ is taken across a zero.  The Gaussian estimate is used only on the interval where positivity follows from the classical ordering of the first positive zeros.  Moreover, $|J_\alpha|^q$ is $C^1$ for $q>1$, so the order differentiation creates no contribution at a zero.  Adjacent Bessel functions have no common positive zero, by the recurrence relations and uniqueness for the Bessel equation.

\subsection{Coverage of the spherical parameters}

For the spherical weight
\[
\beta:=2\nu+1-\nu q,
\qquad
b:=\beta+1=2-\nu(q-2),
\qquad
\delta:=q\left(\nu+\frac12\right)-2\nu-2,
\]
admissibility is exactly $\delta>0$. Theorem~\ref{thm:base-bridge} gives the first comparison for every admissible exponent, independently of the sign of $b$. If $b\le0$, Lemma~\ref{lem:order-derivative} gives the required order monotonicity immediately. If $b>0$, then
\[
2<q<6,
\qquad
0<b<\frac q2,
\qquad
b\le3-\frac q2,
\qquad
\nu+1=\frac{q-b}{q-2}\ge\frac32.
\]
The positive--negative balance in Section~\ref{sec:higher-dimensional} proves
$\partial_\alpha M_\alpha(q,\beta)<0$ for every continuous order
$\alpha\ge\nu+1$.  Integrating this negative derivative over adjacent unit
intervals proves all comparisons after the first mode.  No continuous-order
estimate on $[\nu,\nu+1]$ is needed.

On the circle, the first three positive degrees are treated on $4<q\le6$,
while degrees $k\ge4$ are handled first on $4<q\le5$.  The Landau bound
extends the high-mode comparison to $5\le q<6$, and the previously known
range \eqref{eq:circle-known-large-q} covers $q\ge6$.  No infinite fourth
moment is used at the critical exponent.

\subsection{Normalization changes}

If $\bar\sigma:=\sigma/\omega_{d-1}$ is used both on the source sphere and in the angular norm, and the extension operator is defined with $\bar\sigma$, then
\[
\|f\|_{L^2(\bar\sigma)}
=\omega_{d-1}^{-1/2}\|f\|_{L^2(\sigma)},
\qquad
\widehat{f\bar\sigma}=\omega_{d-1}^{-1}\widehat{f\sigma},
\qquad
\bar C_{d,q}:=\frac{C_{d,q}}{\omega_{d-1}}.
\]
If only the Fourier phase is changed from $e^{-ix\cdot\omega}$ to $e^{-2\pi ix\cdot\omega}$, then
\[
C_{d,q}^{(2\pi)}:=(2\pi)^{-d/q}C_{d,q}
\quad(q<\infty),
\qquad
C_{d,\infty}^{(2\pi)}:=C_{d,\infty}.
\]

\bibliographystyle{amsalpha}

\section*{Acknowledgement}

The author used generative AI tools as an aid in the development and
presentation of this work during the preparation of this manuscript. The
author takes full responsibility for the mathematical content of this
manuscript, including all statements, proofs, computations, and conclusions.

\end{document}